\documentclass[11pt,reqno]{amsart}
\usepackage{fullpage}
\usepackage{xcolor}
\usepackage[T1]{fontenc}                                   
\usepackage{amsfonts}
\usepackage[utf8]{inputenc}                               
\usepackage{comment}                                       
\usepackage{mparhack}                                      
\usepackage{amsmath,amssymb,amsthm,mathrsfs,eucal}                      

\usepackage{mathtools}
\usepackage{booktabs}                                                                          
\usepackage{graphicx,subfig}                                                                
\usepackage{wrapfig}                                                                            
\usepackage[bookmarks=true,colorlinks=true, citecolor=blue]{hyperref}                      
\usepackage{bm}     

\usepackage[normalem]{ulem}

\usepackage{dsfont}
\usepackage{esint}

\newtheorem{thm}{Theorem}[section]
\newtheorem{prop}{Proposition}[section]
\newtheorem{lemma}{Lemma}[section]
\newtheorem{cor}{Corollary}[section]

\theoremstyle{definition}
\newtheorem{definition}{Definition}[section]

\newtheorem{rem}{Remark}[section]

\DeclarePairedDelimiter{\abs}{\lvert}{\rvert}

\DeclarePairedDelimiter{\norm}{\lVert}{\rVert}
\newcommand{\R}{\mathbb{R}}
\newcommand{\N}{\mathbb{N}}
\newcommand{\de}{\partial}
\newcommand{\intd}{\int_{\R^d}}
\renewcommand{\d}{\,\text{d}}

\newcommand{\calP}{\mathcal{P}}

\newcommand{\calS}{\mathcal{S}}

\newcommand{\calW}{\mathcal{W}}
\newcommand{\calF}{\mathcal{F}}
\newcommand{\calG}{\mathcal{G}}
\newcommand{\calH}{\mathcal{H}}
\newcommand{\calD}{\mathcal{D}}

\title{Two-species system with nonlocal interactions driven by Riesz potentials}

\author{Simone Fagioli \and Valeria Iorio}

\begin{document}

\address{Simone Fagioli - DISIM - Department of Information Engineering, Computer Science and Mathematics, University of L'Aquila, Via Vetoio 1 (Coppito)
67100 L'Aquila (AQ) - Italy}
\email{simone.fagioli@univaq.it}

\address{Valeria Iorio - DISIM - Department of Information Engineering, Computer Science and Mathematics, University of L'Aquila, Via Vetoio 1 (Coppito)
67100 L'Aquila (AQ) - Italy}
\email{valeria.iorio1@univaq.it}


\begin{abstract}
This paper investigates a system of nonlocal continuity equations modelling the interaction of two species coupled through Riesz-type potentials. 
The model incorporates self- and cross-interaction kernels of possibly different fractional orders.
By exploiting optimal transportation theory and the theory of gradient flows in Wasserstein spaces, we establish the existence of weak solutions under singularity assumptions on all interaction potentials, provided the cross-interaction ones satisfy a symmetry condition. Our analysis extends previous results available for either single-species equations or multi-species systems with smoother cross-interaction kernels. 
\end{abstract}

\maketitle

\section{Introduction}
This paper focuses on the analysis of the following system of nonlocal continuity equations
\begin{equation}
\label{eq:main_syst_1}
    \begin{dcases}
        \partial_t \rho(t,x) = \mathrm{div}\bigl(\rho(t,x)\,(\nabla K_s \ast \rho(t,x) + \nabla K_q \ast \eta(t,x))\bigr), \\
        \partial_t \eta(t,x) = \mathrm{div}\bigl(\eta(t,x)\,(\nabla K_r \ast \eta(t,x) + \nabla K_q \ast \rho(t,x))\bigr),
    \end{dcases}
\end{equation}
for $x\in\mathbb{R}^d$ and $t\geq 0$, equipped with a given initial datum $(\rho(0,x),\eta(0,x))=(\rho_0(x),\eta_0(x))$
for $x\in\mathbb{R}^d$. The model \eqref{eq:main_syst_1} describes the evolution of two interacting species $\rho$ and $\eta$ coupled through nonlocal interactions driven by the Riesz-type potentials $K_s$, $K_r$ and $K_q$, that is,
\[
    K_s(x) = C_{d,s}\,\abs{x}^{-d+2s}, \qquad s \in \left(0, \min\left\{1,\frac{d}{2}\right\}\right),
\]
where $C_{d,s}$ is a normalisation constant. In particular, $K_s$ and $K_r$ are the self-interaction potentials, modelling the behaviour between agents of the same species with exponents $s$ and $r$, respectively, while $K_q$ is the cross-interaction potential, modelling the mutual interaction between individuals of different species.

The mathematical modelling of collective motion in multi-agent systems appears in a wide range of disciplines, including biology, ecology, robotics, control theory, sociology, and economics. In the last decades, this topic has received considerable attention, and the modelling of bird flocks, fish schools, and insect swarms has been extensively studied by numerous applied mathematicians; see, for example, the classical references, \cite{boi,mogilner,okubo,topaz}.
Among the various possible modelling approaches, particular attention has been devoted to the discrete formulation, considering a discrete collection of $N$ interacting agents (or particles), whose positions $X_1(t), \ldots, X_N(t) \in \mathbb{R}^d$ evolve in time. In a classical dynamical setting, neglecting both inertial or persistence effects (as they are called in cell biology) , the motion of said particles can be formulated via an appropriate Cauchy problem, namely
\begin{equation*}
        \dot{X}_i(t) = -\sum_{j=1}^N \nabla K\bigl(X_i(t)-X_j(t)\bigr),
\end{equation*}
for given initial positions $\overline{X}_i$, $i=1,\ldots,N$, where $K$ plays the role of the interaction potential, prescribing the interaction law between the agents. Typical choices for the interaction potential in these cases include attractive-repulsive Morse potentials
$K(x)= -C_a e^{-\abs{x}/l_a} + C_r e^{-\abs{x}/l_r}$, (where $l_a$ and $l_r$ denote the attractive and repulsive ranges, respectively), attractive-repulsive Gaussian potentials $K(x)= -C_a e^{-\abs{x}^2/l_a} + C_r e^{-\abs{x}^2/l_r}$,  characteristic functions of a set $A$, $K(x)= \alpha \chi_A(x)$ or Manev potential $K(x)=C_M/\abs{x}^2 + C_c/\abs{x}$, cf.\ \cite{manev_1,manev_2,choi_j,difra_i_schm,difra_iorio}. 

It is well known that the above system of interacting particles possesses a continuum counterpart in the nonlocal interaction equation
\begin{equation*}
    \partial_t \rho = \text{div}\left( \rho \, \nabla K \ast \rho \right),
\end{equation*}
where $\rho \coloneqq \rho(t,x)$ denotes the macroscopic density of particles; see \cite{CaCho21,Carrillo2014,Golse2016,HauJab}. The well-posedness of the previous equation is by now well established in classical $L^p$ frameworks, where \emph{blow-up} phenomena are also shown to occur under singularity assumptions on the interaction potential; see \cite{bertozzi1,bertozzi2,bertozzi4,bertozzi3}. Motivated by these considerations, an optimal transport based theoretical framework has been developed, yielding global well-posedness in the Wasserstein space of probability measures under mild assumptions on $K$. Such conditions still include the most relevant examples that exhibit finite-time blow-up, including the attractive Morse potential. Some of the results in \cite{cdfls} extend the theory previously established in \cite{ags}, which applies to the case of smooth interaction potentials.

A natural extension is then to consider the case of Riesz-type potentials
\[
    K_s(x) = C_{d,s}\,\abs{x}^{-d+2s}, \qquad 
    C_{d,s} = \pi^{-d/2} 2^{-2s}\, \Gamma(d/2 - s)\,/\,\Gamma(s),
\]
where $\Gamma$ is the Euler Gamma function; see \cite{adams1999function}. By applying the Fourier transform, one obtains that the Riesz potential is a Fourier multiplier, and in particular
\[
    \widehat{K_s}(\xi)=\abs{\xi}^{-2s},
\]
with the Fourier transform of a function $f$ given by $\hat{f}(\xi)=\int_{\mathbb{R}^d} e^{-ix\cdot \xi}\, f(x)\,dx$.
In this framework, the nonlocal interaction equation can be rewritten as the fractional porous medium equation, a nonlocal diffusion equation driven by the $s$-fractional Laplacian operator, namely
\begin{equation}
    \label{eq:intro_one_species}
    \partial_t \rho = \mathrm{div}\bigl(\rho\, \nabla (-\Delta)^{-s} \rho\bigr),
\end{equation}
where $(-\Delta)^s$ is the $s$-fractional Laplacian on $\mathbb{R}^d$, defined by
\[
    (\widehat{(-\Delta)^s v})(\xi)=\abs{\xi}^{2s}\, \hat{v}(\xi).
\]
The $s$-fractional diffusion equation has been intensively studied in the last three decades. We first mention the seminal paper \cite{caff_vazq_nonlinear_pm}, where suitable initial data are considered. That result has been extended to $L^1$ positive initial data in \cite{caff_serf_vazq_reg_sol_pm} and to non-negative finite measures as initial data in \cite{Serfaty_vazq}. Classical solutions are investigated in \cite{choi_jeong_fractional}. We also mention \cite{mainini_volzone}, in which the authors consider aggregation driven by the Riesz potential and nonlinear diffusion.

Our approach is inspired by the results and techniques in \cite{lis_main_seg}, where the authors show the rigorous construction of non-negative solutions to \eqref{eq:intro_one_species} as trajectories of a gradient flow. More precisely, they use the energy functional 
\[
\mathcal{F}_s(\rho)=\frac{1}{2}\norm{\rho}^2_{\dot{H}^{-s}(\mathbb{R}^d)} \coloneqq \frac{1}{2(2\pi)^d} \int_{\mathbb{R}^d} \abs{\xi}^{-2s}\,\abs{\hat{\rho}(\xi)}^2\d \xi
= \frac{1}{2}\iint_{\mathbb{R}^d\times \mathbb{R}^d} K_s(x-y)\d\rho(x) \d \rho(y),
\]
showing that a solution to the Cauchy problem associated to \eqref{eq:intro_one_species} can be obtained via the minimising movement scheme applied to the functional $\mathcal{F}_s$ in the metric space of probability measures, in the spirit of the \emph{JKO-scheme} originally introduced in \cite{jko}.

According to the previous discussion, model \eqref{eq:main_syst_1} can be rewritten as
\[
    \begin{dcases}
        \partial_t \rho = \mathrm{div}\left[\, \rho\, \nabla\left((-\Delta)^{-s}\rho + (-\Delta)^{-q}\eta\right)\right], \\
        \partial_t \eta = \mathrm{div}\left[\, \eta\, \nabla\left((-\Delta)^{-r}\eta + (-\Delta)^{-q}\rho\right)\right],
    \end{dcases}
\]
and we may formally associate the energy functional
\begin{equation*}
    \begin{aligned}
        \mathcal{F}(\rho,\eta) = {} & \frac{1}{2}\iint_{\mathbb{R}^d \times \mathbb{R}^d} K_s(x-y)\,d\rho(x)\,d\rho(y)
        + \frac{1}{2}\iint_{\mathbb{R}^d \times \mathbb{R}^d} K_r(x-y)\,d\eta(x)\,d\eta(y) \\
        & + \iint_{\mathbb{R}^d \times \mathbb{R}^d} K_q(x-y)\,d\rho(x)\,d\eta(y).
    \end{aligned}
\end{equation*}

Regarding the multi-species setting, \cite{difrafag} shows that the Gradient Flow theory in Wasserstein spaces developed in \cite{ags, cdfls} can be extended to systems under mildly singular assumptions on all interaction kernels, namely Morse-type singularities together with a \emph{symmetry} requirement on the cross-interaction potentials, meaning that mutual interactions are governed by a single potential. When this symmetry assumption is removed, \cite{difrafag} establishes the existence of weak measure-valued solutions by employing a semi-implicit variant of the \emph{JKO-scheme}. In this case, existence can be proved under mildly singular assumptions on the self-interaction potentials and smoothness assumptions on the cross-interaction kernels. This result was further extended in \cite{choi_fagioli_iorio}, where the authors show the existence of solutions to the system with Riesz-type self-interaction potentials and smooth cross-interaction potentials, relying on an iterative adaptation of the argument introduced in \cite{choi_jeong_fractional}.

The present work advances the results mentioned above by establishing the existence of weak solutions in the case where all interaction potentials are of Riesz-type, under the restriction of symmetric cross-interactions, a setting in which one can recover a formal Gradient Flow structure. To the best of the authors' knowledge, no results in the literature cover the case of singular non-symmetric cross-interaction kernels.

The paper is organised as follows: in Section \ref{sec:preliminaries}, we present some preliminary concepts concerning optimal transportation theory and functional analysis, specifically focusing on Fourier transforms and fractional Sobolev spaces; we also state our main existence result in Theorem \ref{thm:intro_main}. Section \ref{sec:jko} is devoted to the variational framework, where we introduce the energy functional and describe the construction of the approximating sequence via the JKO scheme, as well as some a priori estimates performed at the discrete level. Finally, in Section \ref{sec:convergence}, we provide the detailed proof of the convergence of the scheme to the weak solutions to system \eqref{eq:main_syst_1}.

\section{Preliminaries} \label{sec:preliminaries}

In this section, we collect some preliminary concepts and results needed in the sequel, and we also state our main result in Theorem~\ref{thm:intro_main}. We refer the reader to the classical references \cite{ags, bahouri, santambrogio_book, villani}.

\subsection{Optimal transportation theory}
Let $\calP(\R^d)$ be the space of probability measures on $\R^d$, with $d\geq1$. Given $\mu \in \calP(\R^d)$ and a Borel map $T:\R^d \to \R^n$, the \emph{push-forward measure of $\mu$ through the map $T$} is denoted by $\nu \coloneqq T_\# \mu$ and is defined by
\[
    \nu (A) = \mu (T^{-1} (A)), \qquad \mbox{for all Borel sets $A \subset \R^n$.}
\]
The map $T$ is the \emph{transport map} and pushes the measure $\mu$ to the measure $\nu$. 
Let $\calP_2(\R^d)$ be the space of probability measures on $\R^d$ with finite second moments, that is
\[
    \calP_2 (\R^d) = \left\{ f \in \calP(\R^d) \ : \: \mathrm{m}_2[f] < +\infty \right\},
\]
where
\[
\mathrm{m}_2[f]:=\int_{\R^d} \abs{x}^2 \d f(x).
\]
We equip the space $\calP_2 (\R^d)$ with the $2$-Wasserstein distance defined on $\calP_2 (\R^d) \times \calP_2 (\R^d)$ as
\begin{equation}
    \label{eq:_prelim_dist}
    W^2_2(\mu,\nu) \coloneqq \inf_{\boldsymbol{\gamma} \in \Pi (\mu, \nu)} \left\{ \iint_{\R^d \times \R^d} \abs{x-y}^2 \d \boldsymbol{\gamma} (x,y) \right\},
\end{equation}
where $\Pi (\mu, \nu)$ is the class of transport plans between $\mu$ and $\nu$, namely $\boldsymbol{\gamma} \in \Pi (\mu, \nu)$ is the probability measure on $\R^d \times \R^d$ such that $\pi^1 _\# \boldsymbol{\gamma} = \mu$ and $\pi^2\# \boldsymbol{\gamma} = \nu$, where $\pi^i$ is the projector operator on the $i$-th component of the product space.
We then introduce the class of optimal plans between $\mu$ and $\nu$ denoted by $\Pi_o (\mu, \nu)$ as the minimizers of \eqref{eq:_prelim_dist} and we rewrite the $2$-Wasserstein distance as 
\[
    W^2_2(\mu,\nu) = \iint_{\R^d \times \R^d} \abs{x-y}^2 \d \boldsymbol{\gamma} (x,y) , \qquad \boldsymbol{\gamma} \in \Pi_o (\mu, \nu).
\]
Since in this paper we deal with a two species case, we will work on the product space $\calP(\R^d) \times \calP(\R^d)$. We then define the $2$-Wasserstein distance on the product space as
\[
    \calW_2 ^2 ((\mu_1, \mu_2), (\nu_1, \nu_2))=W_2^2 (\mu_1, \nu_1) + W_2^2 (\mu_2, \nu_2),
\]
for $(\mu_1, \mu_2), (\nu_1, \nu_2) \in \calP_2(\R^d) \times \calP_2 (\R^d)$.

\subsection{Fourier transform and Sobolev Spaces}
Let $\calS (\R^d)$ be the Schwartz space, i.e., the set of smooth functions on $\R^d$ whose derivatives are rapidly decreasing. We denote by $\calS' (\R^d)$ the dual space of tempered distributions, that are continuous and linear functionals on $\calS(\R^d)$. The Fourier transform of $f \in \calS (\R^d)$ is defined by
\[
    \hat{f} (\xi) \coloneqq \int_{\R^d} e^{i x \cdot \xi} f(x)\d x.
\]
We collect below the definitions of fractional and homogeneous fractional Sobolev spaces.
\begin{definition}
Let $s\in \R$. The fractional Sobolev space $H^s (\R^d)$ is the space of tempered distributions $f\in \calS' (\R^d)$ such that $\hat{f} \in L^1_{\text{loc}} (\R^d)$ and 
\[
    \norm{f}^2_{H^s(\R^d)} \coloneqq \frac{1}{(2\pi)^d} \intd (1+\abs{\xi}^2)^s \abs{\hat{f}(\xi)}^2 \d \xi < + \infty.
\]
The homogeneous fractional Sobolev space $\dot{H}^s (\R^d)$ is the set of tempered distributions $f \in \calS' (\R^d)$ such that $\hat{f} \in L^1_{\text{loc}} (\R^d)$ and 
\[
    \norm{f}^2_{\dot{H}^s(\R^d)} \coloneqq \frac{1}{(2\pi)^d} \intd \abs{\xi}^{2s} \abs{\hat{f}(\xi)}^2 \d \xi < + \infty.
\]
\end{definition}

Let $s \in (0,1)$, and $f,g\in\dot{H}^s(\R^d)$. We endow $\dot{H}^s(\R^d)$ with the scalar product
\[
    \langle f, g \rangle_s \coloneqq \frac{1}{(2 \pi)^d}\int_{\R^d} \abs{\xi}^{2s} \hat{f} (\xi) \overline{\hat{g} (\xi)} \d \xi.
\]
where $\displaystyle \overline{\hat{f} (\xi)}= \hat{f} (-\xi)$ denotes the conjugate of $f$. We further recall that, by Plancherel formula, for $f, g \in L^2 (\R^d)$, it holds
        \[
            \intd \hat{f} (\xi) \overline{\hat{g}(\xi)} \d \xi = (2 \pi)^d \intd f(x) g(x) \d x. 
        \]
It is easy to check that if $s_1 < s_2$, then $\norm{f}_{H^{s_1}(\R^d)} \leq \norm{f}_{H^{s_2}(\R^d)}$. Moreover, we have that for any $s>0$, $\norm{f}_{\dot{H}^s(\R^d)} \leq \norm{f}_{H^s(\R^d)}$, and for any $s<0$, $\norm{f}_{H^s(\R^d)} \leq \norm{f}_{\dot{H}^s(\R^d)}$.  
Finally, if $s=0$, then $\norm{f}_{L^2(\R^d)} = \norm{f}_{\dot{H}^s(\R^d)} = \norm{f}_{H^s(\R^d)}$.

\begin{rem}
Given $f,g \in \dot{H}^s (\R^d)$, a straightforward computation shows that the scalar product of $f$ and $g$ in $\dot{H}^s (\R^d)$ can be rephrased as
\begin{equation}
    \label{eq:product_H_s}
    \langle f, g \rangle_s = C_{d,s} \iint_{\R^d \times \R^d} \abs{x-y}^{-d-2s} (f(x)-f(y)) (g(x)-g(y)) \d x \d y.
\end{equation}
Indeed, by using the basic property of Fourier transform, Plancherel formula and recalling that $\displaystyle \hat{K}_s (\xi)=\abs{\xi}^{-2s}$, we have 
\begin{align*}
    \langle f, g \rangle_s & = \frac{1}{(2 \pi)^d} \intd \abs{\xi}^{-2(1-s)} \abs{\xi}^2 \hat{f}(\xi) \overline{\hat{g}(\xi)} \d \xi = \frac{1}{(2 \pi)^d} \intd \abs{\xi}^{-2(1-s)} \left( i \xi \hat{f}(\xi) \right) \left(-i \xi \hat{g}(-\xi) \right) \d \xi \\
    & = \frac{1}{(2 \pi)^d} \intd \hat{K}_{1-s} (\xi) \widehat{\nabla f} (\xi) \overline{\widehat{\nabla g} (\xi)} \d \xi = \frac{1}{(2 \pi)^d} \intd \widehat{K_{1-s} \ast \nabla f} (\xi)  \overline{\widehat{\nabla g} (\xi)} \d \xi \\
    & = \intd (K_{1-s} \ast \nabla f )(x) \nabla g(x) \d x  = - \intd (\Delta K_{1-s} \ast f) (x) g (x) \d x,
\end{align*}
Since $\Delta K_{1-s} = - K_{-s}= -C_{d, -s} \abs{x}^{-d-2s}$, we get the thesis.
\end{rem}

In order to keep the notation to a minimum in what follows, we introduce the spaces
   \[
        X^{s,q}_p\coloneqq \dot{H}^{-s} (\R^d) \cap \dot{H}^{-q} (\R^d) \cap \calP_p (\R^d).
   \]

\subsection{Main result}

The goal of this paper is to prove existence of weak solutions to system \eqref{eq:main_syst_1}, that are introduced in the next Definition.
\begin{definition}[Weak solution to \eqref{eq:main_syst_1}] \label{def:weak_sol} Let $0 < s,r,q < \min \{1, \frac{d}{2} \}$, and $(\rho_0, \eta_0) \in X^{s,q}_2 \times X^{r,q}_2$. An absolute continuous curve $(\rho(t, \cdot), \eta(t,\cdot)) : [0,T] \to \calP_2 (\R^d) \times \calP_2 (\R^d)$ is a \emph{weak solution to system \eqref{eq:main_syst_1}} if it satisfies
\begin{align*}
    & \int_0^T \intd \left( \de_t \varphi - \left( \nabla K_s \ast \rho + \nabla K_q \ast \eta \right) \cdot \nabla \varphi \right) \rho \d x \d t = 0, \\
    & \int_0^T \intd \left( \de_t \chi - \left( \nabla K_r \ast \eta + \nabla K_q \ast \rho \right) \cdot \nabla \chi \right) \eta \d x \d t = 0,
\end{align*}
for all $\varphi, \chi \in C_c^\infty ((0,T) \times \R^d)$.
\end{definition}

We state below the main result of the paper.
\begin{thm} \label{thm:intro_main}
    Let $T>0$ be fixed, and consider $0<s,r,q< \min \{1, \frac{d}{2} \}$, fixed exponent with
    \[
        \max \left\{ \frac{s}{2}, \frac{r}{2} \right\} < q < \min \left\{ \frac{s+1}{2}, \frac{r+1}{2} \right\}.
    \]
    Let $(\rho_0, \eta_0) \in X^{s,q}_2\times X^{r,q}_2$. Then, there exists a weak solution $(\rho, \eta)$ to \eqref{eq:main_syst_1} in the sense of Definition \ref{def:weak_sol}. Moreover, the energy dissipation inequality
    \begin{align*}
        \calF(\rho(t), \eta(t)) & +\frac{1}{2}\int_0^t\int_{\R^d}\abs{\nabla \left( K_s \ast \rho (\sigma) + K_q \ast \eta (\sigma) \right)}^2\rho\d x\d \sigma \\ 
        & \quad +  \frac{1}{2}\int_0^t\int_{\R^d}\abs{\nabla \left( K_r \ast \eta (\sigma) +  K_q \ast \rho (\sigma) \right) }^2\eta \d x\d \sigma \leq\calF(\rho_0, \eta_0).
    \end{align*}
    is satisfied for all $t\in[0,T]$. Finally, if $(\rho_0, \eta_0) \in L^p(\R^d)\times L^p(\R^d)$ for some $p\in [1,\infty]$, then 
    \begin{equation*}
            \norm{\rho (t, \cdot)}_{L^p (\R^d)} + \norm{\eta (t, \cdot)}_{L^p (\R^d)}  \leq  \norm{\rho_0}_{L^p (\R^d)} + \norm{\eta_0}_{L^p (\R^d)}.
    \end{equation*}
    for all $t\in[0,T]$.
\end{thm}

The proof of Theorem \ref{thm:intro_main} is divided into several steps and is summarised in Section \ref{sec:convergence}.

\begin{rem}[Choice of $q$]
    The choice  of the range of $q$ will be clear later. We want to underline that, since both $s$ and $r$ belong to $(0,1)$, one can always find a $q$ that satisfies the assumption.
\end{rem}

\begin{rem}[Case with more than two species] The analysis of model \eqref{eq:main_syst_1} can be extend to the case with more than two species. Assume to deal with $N$ species. Denoting by $\rho_i$ the density of the $i$-th species, $K_{s_i}$ the self-interaction potentials and $K_{q_{ij}}$ the cross-interaction ones with the assumption $q_{ij}=q_{ji}$, the system reads
\[
    \de_t \rho_i = \mathrm{div} \left[ \rho_i \nabla \left( (-\Delta)^{-s_i}\rho_i + \sum_{\substack{j=1 \\ j \neq i}}^N (-\Delta)^{-q_{ij}} \rho_j  \right) \right], \qquad i=1, \ldots, N.
\]
In this case, the request on $q_{ij}$ writes
\[
    \max \left\{ \frac{s_i}{2}, \frac{s_j}{2} \right\} < q_{ij} < \min \left\{ \frac{s_i+1}{2}, \frac{s_j+1}{2} \right\},
\]
and it is still satisfied.
\end{rem}

\subsection{Useful inequalities}

In the next proposition we gather some inequalities we will use throughout the paper. Reader can refer to \cite{bahouri}.
\begin{prop}
\label{prop:inequalities}
Let $s_1 < s < s_2$ be given exponents. Then, the following inequalities are true:
\begin{itemize}
    \item[(i)] if $\theta\in(0,1)$ satisfies $s=(1-\theta)s_1 + \theta s_2$, the interpolation inequalities hold
    \begin{equation}
        \label{eq:interp_H_dot}
        \norm{f}_{H^s (\R^d)} \leq \norm{f}^{1-\theta}_{H^{s_1} (\R^d)} \norm{f}^\theta_{H^{s_2} (\R^d)} \qquad \mbox{and} \qquad
        \norm{f}_{\dot{H}^s (\R^d)} \leq \norm{f}^{1-\theta}_{\dot{H}^{s_1} (\R^d)} \norm{f}^\theta_{\dot{H}^{s_2} (\R^d)}.
    \end{equation}
    \item[(ii)] If $\phi \in \calS(\R^d)$ and $f \in H^s (\R^d)$, there is a constant c depending on $d, s$ and $\phi$ such that
    \[
        \norm{\phi f}_{H^s (\R^d)} \leq c \norm{f}_{H^s (\R^d)}.
    \]
    \item[(iii)] If $\phi \in \calS (\R^d)$, and let $\left\{f_n\right\}_{n\in\N}\subset H^{s_2}(\R) $ be such that $\sup_n \norm{f_n}_{H^{s_2}}(\R^d)$ is bounded. Then $\{ \phi f_n \}_{n \in \N}$ is relatively compact in $H^{s_1}(\R)$.
    \item[(iv)] If $s \in (0, \frac{d}{2})$, then, for any $f \in \dot{H}^s (\R^d)$ the fractional Sobolev inequality is satisfied, i.e.,
    \[
        \norm{f}_{L^q(\R^d)} \leq C \norm{f}_{\dot{H}^s (\R^d)},
    \]
    where $q = \frac{2d}{d-2s}$, and the constant $C$ depends on $d$ and $s$.
    \item[(v)] If $f \in \dot{H}^s (\R^d) \cap L^p (\R^d)$, then 
    \begin{equation}
        \label{ineq:intep}
        \norm{f}_{L^q(\R^d)} \leq C^\theta \norm{f}_{L^p(\R^d)}^{1-\theta} \norm{f}_{\dot{H}^s (\R^d)}^\theta,
    \end{equation}
    for $1 \leq p < q < r=\frac{2d}{d-2s}$, where $\theta = \frac{(p-q)r}{(p-r)q}$.
\end{itemize}
\end{prop}

We state here a result that easily follows from equation \eqref{eq:product_H_s}, see \cite[Proposition 2.2]{lis_main_seg}.
\begin{prop}
    \label{prop:product}
   Let $s \in (0,1)$, $f \in \dot{H}^s (\R^d)$, and $F:\R^d \to \R^d$ be given functions.
    \begin{itemize}
        \item[(i)] If $F$ is non-decreasing, then $\langle f, F(f) \rangle_s \geq 0$.
        \item[(ii)] If $F$ is non-decreasing and Lipschitz continuous with Lipschitz constant $L$, then it holds that $F \circ f \in \dot{H}^s (\R^d)$ and
        \[
            \langle f, F(f) \rangle_s \leq L \langle f, f\rangle_s, \qquad \langle F(f), F(f) \rangle_s \leq L \langle F(f), f \rangle_s.
        \]
        \item[(iii)] If, in addition, $f$ is non-negative and $p \in (1,+\infty)$, then the Stroock-Varopoulos inequality holds, that is
        \[
            \langle f, f^p \rangle_s \geq \frac{4p}{(p+1)^2} \norm{f^{(p+1)/2}}^2_{\dot{H}^s (\R^d)}.
        \]
    \end{itemize}
\end{prop}

\section{Energy functional and JKO scheme} \label{sec:jko}

In this section we set the variational framework adopted to prove the existence of a weak solution to system \eqref{eq:main_syst_1}. More precisely, we introduce the energy functional associated to \eqref{eq:main_syst_1} and construct recursively an approximating sequence via the well-known JKO-scheme; see \cite{jko}.

\subsection{Energy functional}

Consider the pair $(\rho, \eta) \in \calP_2(\R^d) \times \calP_2 (\R^d)$. We introduce the energy functional
\begin{equation}
\label{eq:energy_functional_1}
    \begin{aligned}
        \calF (\rho,\eta) = & \frac{1}{2} \iint_{\R^d \times \R^d} K_s (x-y)\d \rho (x) \d \rho (y) + \frac{1}{2} \iint_{\R^d \times \R^d} K_r (x-y)\d \eta (x) \d \eta (y) \\
        & + \iint_{\R^d \times \R^d} K_q (x-y)\d \rho (x) \d \eta (y).
    \end{aligned}
\end{equation}
Recalling \eqref{eq:product_H_s}, in our context it is convenient to rewrite it as
\begin{equation}
\label{eq:energy_functional_2}
        \calF (\rho,\eta) = \frac{1}{2} \norm{\rho}^2_{\dot{H}^{-s} (\R^d)} + \frac{1}{2} \norm{\eta}^2_{\dot{H}^{-r} (\R^d)} + \langle \rho, \eta \rangle_{\dot{H}^{-q}(\R^d)},
\end{equation}
that is possible since a measure in $\calP_2(\R^d)$ is a tempered distribution and its Fourier transform is a $L^1_\mathrm{loc}$ function.

\begin{prop}
\label{prop:basic_propert}
    The energy functional $\calF$ introduced in \eqref{eq:energy_functional_1}-\eqref{eq:energy_functional_2} is defined in 
    \[\calD (\calF) = X_2^{s,q} \times X_2^{r,q}
    \]
    and is non-negative  for every $(\rho,\eta) \in \calP_2 (\R^d) \times \calP_2 (\R^d)$. Moreover, $\calF$ is sequentially lower semi-continuous with respect to the narrow convergence.
\end{prop}
\begin{proof}
    The domain and non-negativity are easy to check. Let $\{ (\rho_n, \eta_n)\}_n \subset \calP_2 (\R^d) \times \calP_2 (\R^d)$ be a sequence such that $\sup_n \calF (\rho_n, \eta_n) < + \infty$, and that converges narrowly to $(\rho,\eta) \in \calP_2 (\R^d) \times \calP_2 (\R^d)$. It holds that
    \begin{align*}
        \calF(\rho_n, \eta_n) & = \frac{1}{2} \iint_{\R^d \times \R^d} K_s (x-y) \d \rho_n (x) \d \rho_n (y) + \frac{1}{2} \iint_{\R^d \times \R^d} K_r (x-y) \d \eta (x) \d \eta (y) \\
        & \qquad + \iint_{\R^d \times \R^d} K_q (x-y) \d \rho_n (x) \d \eta_n (y) \\
        & \geq \frac{1}{2} \iint_{\R^d \times \R^d} K_s (x-y) \d \rho_n (x) \d \rho_n (y) + \frac{1}{2} \iint_{\R^d \times \R^d} K_r (x-y) \d \eta (x) \d \eta (y).
    \end{align*}
    Now, defining $A_n (\xi) \coloneqq \abs{\xi}^{-s} \hat{\rho}_n(\xi)$ and $B_n (\xi) \coloneqq \abs{\xi}^{-r} \hat{\eta}_n (\xi)$, due to the definition of $K_s$ it follows that $\sup_n ( \norm{A_n}_{L^2 (\R^d)} + \norm{B_n}_{L^2 (\R^d)} ) < + \infty$. Since $L^2(\R^d)$ is weakly compact, then both $\{A_n\}_n$ and $\{ B_n\}_n$ converge weakly in $L^2$ to $A \in L^2(\R^d)$ and $B\in L^2 (\R^d)$ respectively. Since both $\rho_n$ and $\eta_n$ converge narrowly to $\rho$ and $\eta$ respectively, then $\hat{\rho}_n$ and $\hat{\eta}_n$ converge pointwise to $\hat{\rho}$ and $\hat{\eta}$ respectively, thus $A_n (\xi) \to \abs{\xi}^{-2s} \hat{\rho} (\xi)$ and $B_n (\xi) \to \abs{\xi}^{-2r} \hat{\eta}(\xi)$ for all $\xi \in \R^d$. Following \cite{lis_main_seg}, we get the statement.
\end{proof}

\subsection{JKO-scheme}
We are now in the position of constructing our approximating sequence. Consider an initial pair $(\rho_0, \eta_0) \in \calP_2(\R^d) \times \calP_2 (\R^d)$ be such that $\calF (\rho_0, \eta_0) < + \infty$. Let $\tau >0$ be fixed. We construct the recursively sequence $\{ (\rho_\tau^n, \eta_\tau^n) \}_{n \in \N} \in \calP_2(\R^d) \times \calP_2 (\R^d)$ as follows: set $(\rho_\tau^0, \eta_\tau^0) = (\rho_0, \eta_0)$, then
\begin{equation}
    (\rho_\tau^{n+1}, \eta_\tau^{n+1}) \in \text{argmin} \bigg\{ \frac{1}{2 \tau} \calW_2^2 ((\rho_\tau^n, \eta_\tau^n), (\rho,\eta)) + \calF(\rho,\eta) \, : \, (\rho,\eta) \in \calP_2(\R^d) \times \calP_2 (\R^d) \bigg\}. \label{eq:jko_scheme}
\end{equation}

The well-posedness of the minimization problem above is an easy consequence of Proposition \ref{prop:basic_propert}, together with standard results in optimal transportation theory, cf.\ \cite{ags,santambrogio_book}. Let $T>0$ be fixed and set $N=\left \lceil \frac{T}{\tau} \right\rceil$. We define the piecewise constant interpolation of the sequence above as 
\begin{equation}
\label{eq:seq_interpol}
    (\rho_\tau (t), \eta_\tau (t))=(\rho_\tau^n, \eta_\tau^n) \qquad t \in ((n-1)\tau, n\tau]
\end{equation}
for $1\leq n \leq N-1$. We now want to establish a compactness result for the family $\{ (\rho_\tau^n, \eta_\tau^n ) \}_{\tau>0}$.

\begin{thm}
    \label{thm:conv_sol}
    Let $T >0$ be fixed, and let $(\rho_0, \eta_0) \in X_2^{s,q} \times X_2^{r,q}$ be a given initial datum. 
    Then, there is a subsequence $\{ (\rho_{\tau_k}, \eta_{\tau_k} ) \}$ of the sequence defined in \eqref{eq:seq_interpol} narrowly converging as $\tau_k \to 0$ to an absolute continuous curve $(\rho,\eta):[0,T] \to \calP_2 (\R^d)\times\calP_2(\R^d)$ uniformly in $t\in [0,T]$.
\end{thm}
\begin{proof}
    Considering \eqref{eq:jko_scheme}, it holds that
    \[
        \frac{1}{2 \tau} \calW_2^2 ((\rho_ \tau^n, \eta_\tau^n), (\rho_\tau^{n+1}, \eta_\tau^{n+1})) \leq \calF (\rho_\tau^n, \eta_\tau^n) - \calF (\rho_\tau^{n+1}, \eta_\tau^{n+1}), 
    \]
    thus
    \begin{equation}
    \label{eq:proof_energy_1}
        \calF (\rho_\tau^{n+1}, \eta_\tau^{n+1} ) \leq \calF (\rho_0, \eta_0)
    \end{equation}
    for all $n \in \N$. 
    Taking the sum over $k$ from $m$ to $n-1$, with $m<n$, we get
    \[
        \frac{1}{2\tau}\sum_{k=m}^{n-1} \calW_2^2 ((\rho_\tau^k, \eta_\tau^k), (\rho_\tau^{k+1}, \eta_\tau^{k+1})) \leq \calF(\rho_\tau^m, \eta_\tau^m) - \calF(\rho_\tau^n, \eta_\tau^n).
    \]
    By the non-negativity of the energy functional $\calF$ and \eqref{eq:proof_energy_1}, we deduce
    \begin{equation}
    \label{eq:proof_prop_2}
        \frac{1}{2\tau}\sum_{k=m}^{n-1} \calW_2^2 ((\rho_\tau^k, \eta_\tau^k), (\rho_\tau^{k+1}, \eta_\tau^{k+1})) \leq \calF(\rho_0 \eta_0) \eqqcolon \Tilde{C}(\rho_0,\eta_0),
    \end{equation}
     where $\Tilde{C}(\rho_0,\eta_0)$ is a finite constant depending on the initial datum.
    Let $m < n$, and consider $0<s<t<T$ with $s\in ((m-1)\tau, m\tau]$ and $t \in ((n-1)\tau, n\tau]$. By using the triangular inequality we compute
    \begin{align*}
        \calW_2^2 ((\rho_\tau(s), \eta_\tau (s)), (\rho_\tau(t), \eta_\tau (t))) & = \calW_2^2 ((\rho_\tau^m, \eta_\tau^m), (\rho_\tau^n, \eta_\tau^n)) \\
        & \leq \left[ \sum_{k=m}^{n-1} \calW_2 ((\rho_\tau^k, \eta_\tau^k), (\rho_\tau^{k+1}, \eta_\tau^{k+1})) \right]^2 \\
        & \leq (n-m) \sum_{k=m}^{n-1} \calW_2^2 ((\rho_\tau^k, \eta_\tau^k), (\rho_\tau^{k+1}, \eta_\tau^{k+1})).
    \end{align*}
    If $s=0$, due to \eqref{eq:proof_prop_2} we obtain
    \[
        \calW_2^2 ((\rho_0, \eta_0), (\rho_\tau (t), \eta_\tau (t)) \leq \overline{C} (\rho_0, \eta_0, T),
    \]
    with $\overline{C}$ a constant depending only on $\rho_0, \eta_0$ and $T$, that is the second moment of $(\rho_\tau(t), \eta_\tau(t))$ is uniformly bounded in time, implying the compactness of the curves $(\rho_\tau, \eta_\tau)$ with respect to the narrow convergence. Furthermore, it holds that $\abs{n-m} \leq \frac{\abs{t-s}}{\tau} + 1$, and then
    \begin{align*}
        \calW_2 ( (\rho_\tau(s), \eta_\tau (s)), (\rho_\tau(t), \eta_\tau (t))) & \leq \abs{n-m}^{1/2} \left[ \sum_{k=m}^{n-1} \calW_2^2 ((\rho_\tau^k, \eta_\tau^k), (\rho_\tau^{k+1}, \eta_\tau^{k+1})) \right]^{1/2} \\
        & \leq C (\sqrt{\abs{t-s}} + \sqrt{\tau}),
    \end{align*}
    with $C>0$. In order to apply \cite[Proposition 3.3.1]{ags}, we notice that
    \[
        \limsup_{\tau \to 0} \calW_2 ( (\rho_\tau(s), \eta_\tau (s)), (\rho_\tau(t), \eta_\tau (t))) \leq C \sqrt{\abs{s-t}},
    \]
    and the function $g(s,t)=C \sqrt{\abs{t-s}}$ is symmetric on $[0,T]\times [0,T]$ and satisfies $g(s,t)\to 0$ as $(s,t)\to (r,r)$ for all $r\in [0,T]$. This concludes the proof. 
    \end{proof}

In what follows we will also make use of the so-called De Giorgi variational interpolation, defined as
\begin{equation}\label{e:degiorgi}
    (\tilde{\rho}_\tau(t), \tilde{\eta}_\tau(t)) \in \text{argmin} \bigg\{ \frac{\calW_2^2 ((\rho_\tau^{n-1}, \eta_\tau^{n-1}), (\rho,\eta))}{2(t-(n-1)\tau)}  + \calF(\rho,\eta) \, : \, (\rho,\eta) \in \calP_2(\R^d) \times \calP_2 (\R^d) \bigg\},
\end{equation}
for $t\in \left((n-1)\tau,n\tau\right]$ and $(\tilde{\rho}_\tau(0), \tilde{\eta}_\tau(0))=(\rho_0, \eta_0)$; see \cite[Section 3.2]{ags}.

\subsection{First flow interchange. The heat flow.} 

We are now ready to prove some regularity results for the piecewise constant curve $\{ (\rho_\tau^n, \eta_\tau^n )\}_{n \in \N}$ by using the by now classical flow interchange technique introduced in \cite{matthes2009family}, and that is reported in Appendix \ref{sec:appendix}. In order to proceed, consider the decoupled system
\begin{equation}
    \begin{dcases}
    \label{eq:syst_aux_1}
        \de_t u_1 = \Delta u_1, \\
        \de_t u_2 = \Delta u_2,
    \end{dcases}
\end{equation}
that is the gradient flow of the functional 
\[
    \calH (u_1, u_2) = \int_{\R^d} u_1 (x) \log u_1 (x) \d x + \int_{\R^d} u_2 (x) \log u_2(x)\d x
\]
in the space $(\calP_2 (\R^d) \times \calP_2 (\R^d), \calW_2^2)$. Let $\calS_t^\calH = (\calS_{1,t}^\calH, \calS_{2,t}^\calH)$ be the semigroup generated by \eqref{eq:syst_aux_1}, i.e., $(u_1(t), u_2(t))= \calS_t^\calH (u_1(0), u_2(0)) = (\calS_{1,t}^\calH u_1(0), \calS_{2,t}^\calH u_2(0))$, for all $t\in [0,T]$.

\begin{prop} \label{prop:est_auxiliary}
    Let $(\rho_0, \eta_0) \in X_2^{s,q}\times X_2^{r,q}$ such that $\calH (\rho_0, \eta_0) < + \infty$. Let $\{(\rho_\tau^n, \eta_\tau^n)\}_{n \in \N}$ be the sequence defined in \eqref{eq:jko_scheme}, and $(\rho_\tau (t), \eta_\tau(t))$ the corresponding piecewise constant interpolation defined in \eqref{eq:seq_interpol}. Then, for all $t \geq 0$ it holds
    \[
        (\rho_\tau (t), \eta_\tau (t) ) \in \dot{H}^{1-s} (\R^d) \times \dot{H}^{1-r}(\R^d).
    \]
    Furthermore, there exists a constant $C>0$ depending only on the dimension $d$ such that for all $T>0$ it holds
    \begin{align*}
    & \int_0^T   \norm{\rho_\tau (t)}^2_{\dot{H}^{1-s}(\R^d)}  \d t \leq \calH (\rho_0, \eta_0) + C \left( 1 + T \calF (\rho_0, \eta_0) + \mathrm{m}_2[\rho_0]+\mathrm{m}_2[\eta_0]\right), \\
    & \int_0^T  \norm{\eta_\tau (t)}^2_{\dot{H}^{1-r}(\R^d)} \d t \leq \calH (\rho_0, \eta_0) + C \left( 1 + T \calF (\rho_0, \eta_0) + \mathrm{m}_2[\rho_0]+\mathrm{m}_2[\eta_0]\right).
    \end{align*}
\end{prop}
\begin{proof}
    Consider the minimising sequence $(\rho_\tau^{n+1}, \eta_\tau^{n+1})_{n \in \N}$ defined in \eqref{eq:jko_scheme}. By construction, we get
    \begin{align*}
        \begin{aligned}
            & \frac{1}{2 \tau} \calW_2^2 ((\rho_\tau^{n+1}, \eta_\tau^{n+1}), (\rho_\tau^n,\eta_\tau^n)) + \calF(\rho_\tau^{n+1},\eta_\tau^{n+1}) \\
            & \quad \leq \frac{1}{2 \tau} \calW_2^2 ((\calS_{1,h}^\calH \rho_\tau^{n+1}, \calS_{2,h}^\calH \eta_\tau^{n+1}), (\rho_\tau^n,\eta_\tau^n)) + \calF(\calS_{1,h}^\calH \rho_\tau^{n+1}, \calS_{2,h}^\calH \eta_\tau^{n+1}).
        \end{aligned}
    \end{align*}
    Then, for all $h>0$, by rearrenging the terms and dividing by $h$ we have
    \begin{align*}
        \begin{aligned}
            & \tau \frac{\calF(\rho_\tau^{n+1},\eta_\tau^{n+1}) - \calF(\calS_{1,h}^\calH \rho_\tau^{n+1}, \calS_{2,h}^\calH \eta_\tau^{n+1})}{h} \\
            & \quad \leq \frac{1}{2h} \left[ \calW_2^2 ((\calS_{1,h}^\calH \rho_\tau^{n+1}, \calS_{2,h}^\calH \eta_\tau^{n+1}), (\rho_\tau^n,\eta_\tau^n)) -  \calW_2^2 ((\rho_\tau^{n+1},    \eta_\tau^{n+1}), (\rho_\tau^n,\eta_\tau^n))\right].
        \end{aligned}
    \end{align*}
    By the definition of dissipation of $\calF$ along $\calS^\calH$, see Lemma \ref{flowinterlem} below, by taking the $\limsup$ at $h\downarrow 0$, we have
    \begin{equation}
    \label{eq:proof_diss_1}
        \tau \mathsf{D}^\calH \calF (\rho_\tau^{n+1}, \eta_\tau^{n+1} ) \leq \frac{1}{2} \frac{d^+}{dt} \left( \calW_2^2 ((\calS_{1,t}^\calH \rho_\tau^{n+1}, \calS_{2,t}^\calH \eta_\tau^{n+1}), (\rho_\tau^n,\eta_\tau^n)) \right) \bigg\lvert_{t=0}.
    \end{equation}
    By using the E.V.I. inequality \eqref{eq:evi} for the heat-flow we can bound
    \[
    \label{eq:proof_diss_2}
        \tau \mathsf{D}^\calH \calF (\rho_\tau^{n+1}, \eta_\tau^{n+1} ) \leq \calH (\rho_\tau^n,\eta_\tau^n) - \calH (\rho_\tau^{n+1},\eta_\tau^{n+1}).
    \]
    We now aim to produce a bound from below for the  dissipation term. Remember that
    \[
        \mathsf{D}^\calH \calF (\rho_\tau^{n+1}, \eta_\tau^{n+1}) = \limsup_{h \downarrow 0} \int_0^1 \left( - \frac{d}{dz} \bigg\lvert_{z=ht} \calF (\calS_{1,z}^\calH \rho_\tau^{n+1}, \calS_{2,z}^\calH \eta_\tau^{n+1}) \right) \d t.
    \]
Recalling that in Fourier variable it holds
    \[
        \de_t \widehat{\calS_{1,t}^\calH \rho_\tau^n} (\xi) = -\abs{\xi}^2 \widehat{\calS_{1,t}^\calH \rho_\tau^n} (\xi) \qquad \mbox{\and} \qquad \de_t \widehat{\calS_{2,t}^\calH \eta_\tau^n} (\xi) = -\abs{\xi}^2 \widehat{\calS_{2,t}^\calH \eta_\tau^n} (\xi)
    \]
    in $(0,T) \times \R^d$, then the time derivative of the energy functional $\calF$ along the $\calS^\calH$ is
    \begin{align*}
        \frac{d}{dt} \calF(\calS_{1,t}^\calH \rho_\tau^n, \calS_{2,t}^\calH \eta_\tau^n) 
        & = -\frac{1}{(2 \pi)^d} \intd \abs{\xi}^{2(1-s)} \abs{\widehat{\calS_{1,t}^\calH \rho_\tau^n} (\xi)}^2 \d \xi \\
        & \quad -\frac{1}{(2 \pi)^d} \intd \abs{\xi}^{2(1-r)} \abs{\widehat{\calS_{2,t}^\calH \eta_\tau^n} (\xi)}^2 \d \xi \\
        & \quad - \frac{2}{(2 \pi)^d} \intd \abs{\xi}^{2(1-q)} \widehat{\calS_{1,t}^\calH \rho_\tau^n} (\xi) \overline{\widehat{\calS_{2,t}^\calH \eta_\tau^n} (\xi)}\d \xi \\
        & = -\norm{\calS_{1,t}^\calH \rho_\tau^n}^2_{\dot{H}^{1-s}(\R^d)} -\norm{\calS_{2,t}^\calH \eta_\tau^n}^2_{\dot{H}^{1-r}(\R^d)} - 2 \left\langle \calS_{1,t}^\calH \rho_\tau^n, \calS_{2,t}^\calH \eta_  \tau^n \right\rangle_{1-q}.
    \end{align*}

Since flows $\calS_{1,t}^\calH \rho_\tau^n$ and $\calS_{2,t}^\calH \eta_\tau^n$ are solutions to the decoupled heat equations \eqref{eq:syst_aux_1}, they are non-negative and we can express them as
    \[
        \calS_{1,t}^\calH \rho_\tau^n= \Gamma_t \ast \rho_\tau^n, \qquad \mbox{and} \qquad \calS_{2,t}^\calH \eta_\tau^n= \Gamma_t \ast \eta_\tau^n,
    \]
    where $\Gamma_t$ is the heat kernel at time $t$. We can then estimate the $\dot{H}^{1-s}$-norm of $\calS_{1,t}^\calH \rho_\tau^n$ as follows
    \begin{align*}
        \norm{\calS_{1,t}^\calH \rho_\tau^n}_{\dot{H}^{1-s}(\R^d)} = \intd \abs{\xi}^{2(1-s)} \abs{\widehat{\calS_{1,t}^\calH \rho_\tau^n}(\xi)}^2 \d \xi & = \intd \abs{\xi}^{2(1-s)} \abs{\hat{\Gamma_t} (\xi)}^2 \abs{\hat{\rho}_\tau^n (\xi)} ^2\d \xi \\
        & \leq \max_\xi \abs{\xi}^2 \abs{\hat{\Gamma_t} (\xi)}^2 \intd \abs{\xi}^{-2s} \abs{\hat{\rho}_\tau^n (\xi)} ^2\d \xi \\
        & = C(t) \norm{\rho_\tau^n}_{\dot{H}^{-s}(\R^d)} \\
        & \leq 2C(t) \calF (\rho_\tau^n, \eta_\tau^n) \\
        & \leq 2C(t) \calF (\rho_0, \eta_0),
    \end{align*}
    where we used \eqref{eq:proof_energy_1}, with a positive constant $C(t)$ depending on time.
    Similarly, we can bound
    \[
        \norm{\calS_{2,t}^\calH \eta_\tau^n}_{\dot{H}^{1-r}(\R^d)} \leq 2C(t) \calF (\rho_0, \eta_0).
    \]
    We now want to estimate the scalar product between $\calS_1^\calH\rho_\tau^n$ and $\calS_2^\calH\eta_\tau^n$ in $\dot{H}^{1-q}(\R^d)$. We compute
    \begin{align*}
        \langle \calS_{1,t}^\calH\rho_\tau^n, \calS_{2,t}^\calH\eta_\tau^n \rangle_{1-q} & = 
        \intd \abs{\xi}^{2(1-q)} \widehat{\calS_{1,t}^\calH \rho_\tau^n} (\xi) \overline{\widehat{\calS_{2,t}^\calH \eta_\tau^n} (\xi)} \d \xi \\ 
        & = \intd \abs{\xi}^{2(1-q)} \widehat{\Gamma_t \ast \rho_\tau^n} (\xi) \overline{\widehat{ \Gamma_t \ast \eta_\tau^n} (\xi)} \d \xi \\
        & = \intd \abs{\xi}^{2(1-q)} \abs{\hat{\Gamma}_t (\xi )}^2 \hat{\rho}_\tau^n (\xi) \overline{ \hat{\eta}_\tau^n (\xi)} \d \xi \\
        & \leq \max_\xi \abs{\xi}^2 \abs{\hat{\Gamma}_t (\xi)}^2 \intd \abs{\xi}^{-2q} \hat{\rho}_\tau^n (\xi) \overline{ \hat{\eta}_\tau^n (\xi)} \d \xi \\
        & \leq C(t) \langle \rho_\tau^n, \eta_\tau^n \rangle_{\dot{H}^{-q}(\R^d)}.
    \end{align*}
    Therefore, $(\calS_{1,t}^\calH \rho_\tau^n, \calS_{2,t}^\calH \eta_\tau^n) \in (\dot{H}^{1-s} (\R^d) \cap \dot{H}^{1-q} (\R^d) ) \times (\dot{H}^{1-r} (\R^d) \cap \dot{H}^{1-q} (\R^d) ).$ Thus, the map $t \mapsto \calF (\calS_{1,t}^\calH \rho_\tau^n, \calS_{2,t}^\calH \eta_\tau^n)$ is differentiable in $(0, + \infty)$.
    We now want to prove that that $t \mapsto \calF (\calS_{1,t}^\calH \rho_\tau^n, \calS_{2,t}^\calH \eta_\tau^n)$ in continuous in $t=0$. We first notice that $0 \leq \hat{\Gamma} \leq 1$, and thus $\widehat{\Gamma \ast \rho_\tau^n} (\xi) = \hat{\Gamma}(\xi)\hat{\rho}_\tau^n (\xi) \leq \hat{\rho}_\tau^n (\xi)$, and $\abs{\widehat{\Gamma \ast \rho_\tau^n} (\xi)}^2 = \abs{\hat{\Gamma} (\xi) \hat{\rho}_\tau^n (\xi)}^2 \leq \abs{\hat{\rho}_\tau^n (\xi)}^2$. It follows that $\calF(\calS_1^\calH \rho_\tau^n, \calS_2^\calH \eta_\tau^n) \leq \calF (\rho_\tau^n, \eta_\tau^n)$, and due to the lower semi-continuity of $\calF$ proved in Proposition \ref{prop:basic_propert}, we deduce the continuity at $0$. 
    We can now apply the mean value Theorem, obtaining that there is $\xi (t) \in (0,t)$ such that
    \begin{align*}
        & \frac{\calF(\rho_\tau^n, \eta_\tau^n) - \calF(\calS_{1,t}^\calH\rho_\tau^n, \calS_{2,t}^\calH\eta_\tau^n)}{t} \\
        & \quad = \norm{\calS_{1,\xi(t)}^\calH \rho_\tau^{n+1}}^2_{\dot{H}^{1-s}(\R^d)} + \norm{\calS_{2,\xi(t)}^\calH \eta_\tau^{n+1}}^2_{\dot{H}^{1-r}(\R^d)} + 2 \left\langle \calS_{1,\xi(t)}^\calH \rho_\tau^{n+1}, \calS_{2,\xi(t)}^\calH \eta_\tau^{n+1} \right\rangle_{1-q}.
    \end{align*}
    By using the non-negativity of the product space in $H^{1-q}$ and the lower semi-continuity of the norm, we arrive at
    \[
    \mathsf{D}^\calH \calF (\rho_\tau^n, \eta_\tau^n) \ge \norm{\rho_\tau^n}^2_{\dot{H}^{1-s}(\R^d)} + \norm{\eta_\tau^n}^2_{\dot{H}^{1-r}(\R^d)}.
    \]
    By Lemma \ref{flowinterlem}, we get
    \[
        \norm{\rho_\tau^n}^2_{\dot{H}^{1-s}(\R^d)} + \norm{\eta_\tau^n}^2_{\dot{H}^{1-r}(\R^d)} \leq \frac{\calH (\rho_\tau^{n-1}, \eta_\tau^{n-1}) - \calH(\rho_\tau^n, \eta_\tau^n)}{\tau}.
    \]
    Now let $T>0$. The estimate above implies that
    \[
        \int_0^T \norm{\rho_\tau (t)}^2_{\dot{H}^{1-s}(\R^d)} \d t \leq \sum_{k=0}^{N-1} \tau  \norm{\rho_\tau^n}^2_{\dot{H}^{1-s}(\R^d)}   \leq \calH (\rho_0, \eta_0)-\calH(\rho_\tau^N, \eta_\tau^N).
    \]
    Following \cite[Lemma 2.9]{difra_iorio}, one can prove that for an arbitrary $\rho \in \calP_2 (\R^d)$ there is a constant $C>0$ depending only on the dimension $d$ such that
    \begin{equation}
        \label{eq:log}
        \intd \rho (\log \rho + C(\abs{x}^2 +1)\d x \geq 0.
    \end{equation}
    We then get
    \[ 
        -\calH(\rho_\tau^N, \eta_\tau^N) \leq C \left( 1+ \mathrm{m}_2\left[\rho_\tau^N\right] + \mathrm{m}_2\left[\eta_\tau^N\right] \right).
    \]
    Due to the scheme \eqref{eq:jko_scheme}, we also have that
    \[
        \calF (\rho_\tau^N, \eta_\tau^N) + \frac{1}{2} \sum_{k=0}^{N-1} \tau \frac{\calW_2^2 (\rho_\tau^n, \eta_\tau^n), (\rho_\tau^{n+1}, \eta_\tau^{n+1})}{\tau^2} \leq \calF (\rho_0, \eta_0),
    \]
    and thus we compute 
    \begin{align*}
      \mathrm{m}_2\left[\rho_\tau^N\right] + \mathrm{m}_2\left[\eta_\tau^N\right] 
        & = \calW_2^2 ( (\rho_\tau^N,\eta_\tau^N),(\delta_0, \delta_0) ) \\
        & \leq \left( \sum_{k=0}^{N-1} \calW_2 ( (\rho_\tau^k, \eta_\tau^k), (\rho_\tau^{k+1}, \eta_\tau^{k+1})) + \calW_2 ((\rho_0, \eta_0), (\delta_0, \delta_0)) \right)^2 \\
        & \leq 2 \left( \sum_{k=0}^{N-1} \tau \frac{\calW_2 ( (\rho_\tau^k, \eta_\tau^k), (\rho_\tau^{k+1}, \eta_\tau^{k+1}))}{\tau} \right)^2 + 2  \calW_2^2 ((\rho_0, \eta_0), (\delta_0, \delta_0)) \\
        & \leq 2 N \tau \sum_{k=0}^{N-1} \tau \frac{\calW_2^2 ( (\rho_\tau^k, \eta_\tau^k), (\rho_\tau^{k+1}, \eta_\tau^{k+1}))}{\tau^2} + 2  \calW_2^2 ((\rho_0, \eta_0), (\delta_0, \delta_0)) \\
        & \leq 4 T \calF (\rho_0, \eta_0) + 2  \calW_2^2 ((\rho_0, \eta_0), (\delta_0, \delta_0)),
    \end{align*}
    where we used the relations $\calF \geq 0$ and $N\tau = T$. Combining all the estimates, by renaming the constant $C$, we arrive at
    \[
        \int_0^T \norm{\rho_\tau (t)}^2_{\dot{H}^{1-s}(\R^d)} \d t \leq \calH (\rho_0, \eta_0) + C\left( 1 + T \calF (\rho_0, \eta_0) +\mathrm{m}_2\left[\rho_0\right] + \mathrm{m}_2\left[\eta_0\right] \right).
    \]
    By performing a similar argument for the second species, we complete the proof.
\end{proof}

\begin{cor}\label{cor:degiorgireg}
Under the assumptions of Proposition \ref{prop:est_auxiliary}, for any $t\in(0, T]$ we have
\begin{equation*}
    (\tilde{\rho}_\tau(t), \tilde{\eta}_\tau(t))\in \dot{H}^{1-s} (\R^d) \times \dot{H}^{1-r}(\R^d),
\end{equation*}
where $(\tilde{\rho}_\tau(t), \tilde{\eta}_\tau(t))$ is the De Giorgi variational interpolation defined in \eqref{e:degiorgi}.
\end{cor}

\subsection{Second flow interchange and \texorpdfstring{$L^\infty$}{L-infinity} estimate.}

In this subsection, we want to investigate a regularity result about the solutions to system \eqref{eq:main_syst_1}. In particular, our aim is to prove that if the initial datum is in $(\calP (\R^d) \cap L^p(\R^d))^2$, with $p \in (1, + \infty]$, then the solution $(\rho, \eta)$ keeps this regularity.

Consider the decoupled system
\begin{equation}
    \label{eq:auxiliary_p}
    \begin{dcases}
        \de_t u_1 = \Delta u_1^p + \varepsilon \Delta u_1, \\
        \de_t u_2 = \Delta u_2^p + \varepsilon \Delta u_2,
    \end{dcases}
\end{equation}
that can be seen as the gradient flow of
\[
    \calG(u_1, u_2) = \frac{1}{p-1} \intd [u_1(x)^p + u_2(x)^p]\d x + \varepsilon \intd [u_1(x) \log u_1(x)\d x + u_2(x) \log u_2(x) ]\d x
\]
in the space $(\calP_2(\R^d)\times\calP_2(\R^d), \calW_2)$, with $p \in (1, + \infty)$, and $\varepsilon >0$.
We set $\calS^\calG = (\calS_1^\calG, \calS_2^\calG)$ the semigroup generated be \eqref{eq:auxiliary_p}.

\begin{prop}
    Let $T >0$ and $p \in (1, + \infty)$. Assume $(\rho_0, \eta_0) \in \calD (\calF)$ such that $\calG (\rho_0, \eta_0) < + \infty$. Then, the piecewise constant curve $(\rho_\tau, \eta_\tau)$ defined in \eqref{eq:seq_interpol} fulfils
    \[
        \norm{\rho_\tau}_{L^\infty ((0,T); L^p (\R^d))} \norm{\eta_\tau}_{L^\infty ((0,T); L^p (\R^d))} \leq \norm{\rho_0}_{L^p(\R^d)} + \norm{\eta_0}_{L^p (\R^d)},
    \]
    and the limit curve $(\rho, \eta) \in L^\infty ((0,T), L^p(\R^d))$. 
    Furthermore, a similar estimate holds for the case $p=+\infty$.
\end{prop}
\begin{proof}
Considering the Fourier transform of the coordinates of the semigroup $\calS^\calG$, we have
\begin{align*}
    & \de_t \widehat{\calS_{1,t}^\calG \rho_\tau^k} (\xi) = - \abs{\xi}^2 \left[ \widehat{ (\calS_{1,t}^\calG \rho_\tau^k)^p} (\xi) + \varepsilon \widehat{\calS_{1,t}^\calG \rho_\tau^k} (\xi) \right], \\
    & \de_t \widehat{\calS_{2,t}^\calG \eta_\tau^k} (\xi) = - \abs{\xi}^2 \left[ \widehat{(\calS_{2,t}^\calG \eta_\tau^k})^p (\xi) +  \varepsilon \widehat{\calS_{2,t}^\calG \eta_\tau^k} (\xi)\right].
\end{align*}
Following \cite[Lemma 4.7]{lis_main_seg}, we have that both $\calS_{1,t}^\calG \rho_\tau^k$ and $\calS_{2,t}^\calG \eta_\tau^k$ are smooth, bounded and strictly positive for $t >0$; in particular 
\[
    \calS_{1,t}^\calG \rho_\tau^k \in H^{1-s}(\R^d) \cap H^{1-q}(\R^d) \qquad \mbox{and} \qquad \calS_{2,t}^\calG \eta_\tau^k \in H^{1-r}(\R^d) \cap H^{1-q}(\R^d).
\]
Arguing similarly as in Proposition \ref{prop:est_auxiliary}, we have that $\rho_\tau^k \in \dot{H}^{1-s} (\R^d)$ and $\eta_\tau^k \in \dot{H}^{1-r} (\R^d)$, and both $\rho_\tau^k$ and $\eta_\tau^k$ belong to $L^1(\R^d)$ by construction. By using inequality \eqref{ineq:intep} we deduce
\[
    \rho_\tau^k \in L^2(\R^d), \qquad \mbox{and} \qquad \eta_\tau^k \in L^2(\R^d).
\]
Furthermore, due to the monotonic increasing property of the map $f \mapsto f^p$ for $t>0$ and $p>1$, we get
\begin{align*}
    & \norm{\calS_{1,t}^\calG \rho_\tau^k} _{L^2 (\R^d)} \leq \norm{\rho_\tau^k}_{L^2(\R^d)}, \\
    & \norm{\calS_{2,t}^\calG \eta_\tau^k} _{L^2 (\R^d)} \leq \norm{\eta_\tau^k}_{L^2(\R^d)}.
\end{align*}
We obtain
\begin{equation}
    \label{eq:reg_flow}
    \calS_{1,t}^\calG \rho_\tau^k \in \dot{H}^{1-s}(\R^d)\cap \dot{H}^{1-q}(\R^d) \qquad \mbox{and} \qquad \calS_{2,t}^\calG \eta_\tau^k \in \dot{H}^{1-r}(\R^d)\cap \dot{H}^{1-q}(\R^d).
\end{equation}
Following the same strategy as in the Proposition \ref{prop:est_auxiliary}, we have
\begin{equation}
    \label{eq:proof_2aux_1}
        \tau \mathsf{D}^\calG \calF (\rho_\tau^{n+1}, \eta_\tau^{n+1} ) \leq \calG (\rho_\tau^n,\eta_\tau^n) - \calG (\rho_\tau^{n+1},\eta_\tau^{n+1}).
    \end{equation}
    and
    \begin{equation}
        \label{eq:proof_2aux_2}
        \mathsf{D}^\calG \calF (\rho_\tau^{n+1}, \eta_\tau^{n+1}) = \limsup_{h \downarrow 0} \int_0^1 \left( - \frac{d}{dz} \bigg\lvert_{z=ht} \calF (\calS_{1,z}^\calG \rho_\tau^{n+1}, \calS_{2,z}^\calG \eta_\tau^{n+1}) \right) \d t.
    \end{equation}
We then compute
\begin{align*}
    \frac{d}{dt}& \calF (\calS_{1,t}^\calG \rho_\tau^k, \calS_{2,t}^\calG \eta_\tau^k) \\
    & = -\frac{1}{(2\pi)^d} \intd \abs{\xi}^{-2(s-1)} \widehat{(\calS_{1,t}^\calG \rho_\tau^k)^p} (\xi) \overline{ \widehat{\calS_{1,t}^\calG \rho_\tau^k} (\xi) }\d \xi - \frac{1}{(2 \pi)^d} \varepsilon \intd \abs{\xi}^{-2(s-1)} \widehat{ \calS_{1,t}^\calG \rho_\tau^k } (\xi) \overline{ \widehat{\calS_{1,t}^\calG \rho_\tau^k} (\xi) }\d \xi \\
    & \quad -\frac{1}{(2\pi)^d} \intd \abs{\xi}^{-2(r-1)} \widehat{(\calS_{2,t}^\calG \eta_\tau^k)^p} (\xi) \overline{ \widehat{\calS_{2,t}^\calG \eta_\tau^k} (\xi) }\d \xi - \frac{1}{(2 \pi)^d} \varepsilon \intd \abs{\xi}^{-2(r-1)} \widehat{ \calS_{2,t}^\calG \eta_\tau^k } (\xi) \overline{ \widehat{\calS_{2,t}^\calG \eta_\tau^k} (\xi) }\d \xi \\
    & \quad - \frac{1}{(2 \pi)^d} \intd \abs{\xi}^{-2(q-1)} \overline{\widehat{ \calS_{2,t}^\calG \eta_\tau^k} (\xi) } \left( \widehat{ (\calS_{1,t}^\calG \rho_\tau^k)^p } (\xi) + \varepsilon \widehat{ \calS_{1,t}^\calG \rho_\tau^k } (\xi) \right) \d \xi \\
    & \quad - \frac{1}{(2 \pi)^d} \intd \abs{\xi}^{-2(q-1)} \widehat{ \calS_{1,t}^\calG \rho_\tau^k} (\xi)  \left( \overline{ \widehat{ (\calS_{2,t}^\calG \eta_\tau^k)^p } (\xi) } + \varepsilon \overline{ \widehat{ \calS_{2,t}^\calG \eta_\tau^k } (\xi) } \right) \d \xi \\
    & = - \left\langle (\calS_{1,t}^\calG \rho_\tau^k )^p, \calS_{1,t}^\calG \rho_\tau^k \right\rangle_{1-s} - \left\langle (\calS_{2,t}^\calG \eta_\tau^k )^p, \calS_{2,t}^\calG \eta_\tau^k \right\rangle_{1-r}\\
    & \quad - \varepsilon \left\langle \calS_{1,t}^\calG \rho_\tau^k, \calS_{1,t}^\calG \rho_\tau^k \right\rangle_{1-s} - \varepsilon \left\langle \calS_{2,t}^\calG \eta_\tau^k, \calS_{2,t}^\calG \eta_\tau^k \right\rangle_{1-r} \\
    & \quad - \left\langle \calS_{2,t}^\calG \eta_\tau^k , ( \calS_{1,t}^\calG \rho_\tau^k)^p \right\rangle_{1-q} - \left\langle \calS_{1,t}^\calG \rho_\tau^k, (\calS_{2,t}^\calG \eta_\tau^k)^p \right\rangle_{1-q} \\
    & \quad- \varepsilon \left\langle \calS_{2,t}^\calG \eta_\tau^k , \calS_{1,t}^\calG \rho_\tau^k \right\rangle_{1-q} - \varepsilon \left\langle \calS_{1,t}^\calG \rho_\tau^k, \calS_{2,t}^\calG \eta_\tau^k \right\rangle_{1-q}.
\end{align*}
Since the function $v \mapsto v^p$ is non-decreasing and locally Lipschitz as $p>1$, due to \eqref{eq:reg_flow} and Proposition \ref{prop:product} we get
\[
    (\calS_{1,t}^\calG \rho_\tau^k )^p \in \dot{H}^{1-s}(\R^d) \cap \dot{H}^{1-q}(\R^d), \quad \mbox{and} \quad (\calS_{2,t}^\calG \eta_\tau^k )^p \in \dot{H}^{1-r}(\R^d) \cap \dot{H}^{1-q}(\R^d),
\]
and, furthermore, the terms
\[
    \left\langle (\calS_{1,t}^\calG \rho_\tau^k )^p, \calS_{1,t}^\calG \rho_\tau^k \right\rangle_{1-s}, \quad \left\langle (\calS_{2,t}^\calG \eta_\tau^k )^p, \calS_{2,t}^\calG \eta_\tau^k \right\rangle_{1-r}, \quad 
    \left\langle \calS_{2,t}^\calG \eta_\tau^k , ( \calS_{1,t}^\calG \rho_\tau^k)^p \right\rangle_{1-q}, \quad \left\langle \calS_{1,t}^\calG \rho_\tau^k, (\calS_{2,t}^\calG \eta_\tau^k)^p \right\rangle_{1-q}
\]
are non-negative. More, due to \eqref{eq:reg_flow}, the quantities
\[
    \left\langle \calS_{1,t}^\calG \rho_\tau^k, \calS_{1,t}^\calG \rho_\tau^k \right\rangle_{1-s}, \quad \left\langle \calS_{2,t}^\calG \eta_\tau^k, \calS_{2,t}^\calG \eta_\tau^k \right\rangle_{1-r}, \quad \left\langle \calS_{2,t}^\calG \eta_\tau^k , \calS_{1,t}^\calG \rho_\tau^k \right\rangle_{1-q}
\]
are finite and non-negative. We then have 
\[
    \frac{d}{dt} \calF (\calS_{1,t}^\calG \rho_\tau^k, \calS_{2,t}^\calG \eta_\tau^k) \leq 0.
\]
Combining this with \eqref{eq:proof_2aux_1} and \eqref{eq:proof_2aux_2} we get
\[
    \calG (\rho_\tau^{n+1}, \eta_\tau^{n+1}) \leq \calG (\rho_\tau^n, \eta_\tau^n), 
\]
i.e., 
\[
    \calG (\rho_\tau^n, \eta_\tau^n) \leq \calG(\rho_0, \eta_0),
\]
for all $n \in \N$.
Now, since by \eqref{eq:log} we control the logarithmic part of $\calG$, by letting $\varepsilon \downarrow 0^+$, for all $p \in (1, + \infty)$ we deduce
\[
    \intd \left[ \rho_\tau^n (x) ^p + \eta_\tau^n (x) ^p \right] \d x \leq \intd \left[ \rho_0 (x) ^p + \eta_0 (x)^p \right] \d x, 
\]
that implies
\[
    \intd \left[ \rho_\tau (x,t) ^p + \eta_\tau (x,t) ^p \right] \d x \leq \intd \left[ \rho_0 (x) ^p + \eta_0 (x)^p \right] \d x, 
\]
for all $p \in (1, +\infty)$ and all $t \in [0,T]$. A similar estimate can be obtained for $p=+\infty$. Indeed,
\begin{align*}
    \norm{\rho_\tau (t, \cdot)}_{L^\infty (\R^d)} + \norm{\eta_\tau (t, \cdot)}_{L^\infty (\R^d)} & \leq \limsup_{p \to +\infty} \left[ \norm{\rho_\tau (t, \cdot)}_{L^p (\R^d)} + \norm{\eta_\tau (t, \cdot)}_{L^p (\R^d)}\right] \\
    & \leq \limsup_{p \to + \infty} \left[ \norm{\rho_0}_{L^p (\R^d)} + \norm{\eta_0}_{L^p (\R^d)} \right] \\
    & \leq \limsup_{p \to + \infty} \left[ \norm{\rho_0}^\frac{p-1}{p}_{L^\infty(\R^d)} \norm{\rho_0}^\frac{1}{p}_{L^1(\R^d)} + \norm{\eta_0}^\frac{p-1}{p}_{L^\infty(\R^d)} \norm{\eta_0}^\frac{1}{p}_{L^1(\R^d)} \right] \\
    & = \norm{\rho_0}_{L^\infty (\R^d)} + \norm{\eta_0}_{L^\infty (\R^d)}.
\end{align*}
We conclude that the subsequence $\{ (\rho_{\tau_k}, \eta_{\tau_k}) \}_k$ provided in Theorem \ref{thm:conv_sol} is uniformly bounded in $L^\infty ([0,T], L^p (\R^d))^2$, and then it admits a subsequence $\{ (\rho_{\tau_{k'}}, \eta_{\tau_{k'}}) \}_{k'}$ converging weakly for $p \in (1,+\infty)$ and weakly-$\ast$ for $p=+\infty$ in $L^p ([0,T], \R^d)^2$. Such limit coincides with $(\rho,\eta)$ on $[0,T]$ due to the narrow convergence proved in Theorem \ref{thm:conv_sol} and satisfies the same estimate.
\end{proof}

\section{Convergence to weak solutions} \label{sec:convergence}

In this section we deal with the consistency of the scheme, proving the convergence towards a weak solution to \eqref{eq:main_syst_1}. This result is made up by several steps that we included into the two following lemmas for the reader's convenience. Notice that we present here a detailed proof for the first species. A similar argumentation can be applied for the second species.
\begin{lemma}
    \label{lemma:convergence}
    Let $T>0$ be fixed and consider the initial datum $(\rho_0, \eta_0) \in X_2^{s,q}\times X_2^{r,q}$. Let $(\rho_\tau, \eta_\tau)$ be the piecewise constant interpolation defined in \eqref{eq:seq_interpol} and let  $(\rho_{\tau_n}, \eta_{\tau_n})_n$ be a subsequence that converges to the limit curve $(\rho, \eta)$ as $\tau_n \to 0$, as given in Theorem \ref{thm:conv_sol}. Then
    \begin{itemize}
        \item[(i)] $(\rho, \eta) \in H^{1-s}(\R^d) \times H^{1-r} (\R^d)$, and $\rho_{\tau_n} \to \rho$, and $\eta_{\tau_n} \to \eta$ strongly in $L^2 ((0,T); L^2_\mathrm{loc} (\R^d))$ as $n \to + \infty$. 
        \item[(ii)] Furthermore, if
        \[
            \max \left\{ \frac{s}{2}, \frac{r}{2} \right\} < q < \min \left\{ \frac{s+1}{2}, \frac{r+1}{2} \right\},
        \]
        then $\nabla K_s \ast \rho_{\tau_n} + \nabla K_q \ast \eta_{\tau_n} \rightarrow \nabla K_s \ast \rho + \nabla K_q \ast \eta$, and $ \nabla K_r \ast \eta_{\tau_n} + \nabla K_q \ast \rho_{\tau_n} \rightarrow \nabla K_r \ast \eta + \nabla K_q \ast \rho$ weakly $L^2((0,T), L^2 (\R^d))$ as $n \to + \infty$.
    \end{itemize}
\end{lemma}
\begin{proof}
    \underline{Statement (i).} The proof of this first point follows the proof in \cite[Lemma 6.1]{lis_main_seg} and we report it for completeness. We start noticing that the interpolation inequality \eqref{eq:interp_H_dot} for $\theta = s$ reads
    \[
        \norm{\rho_\tau (t)}_{L^2(\R^d)} \leq \norm{\rho_\tau (t)}^{1-s}_{\dot{H}^{-s}(\R^d)} \norm{\rho_\tau (t)}^s_{\dot{H}^{1-s}(\R^d)}.
    \]
    By Holder inequality, estimate \eqref{eq:proof_energy_1} and Proposition \ref{prop:est_auxiliary} we get
    \begin{align*}
        \int_0^T \norm{\rho_\tau (t)}^2_{L^2(\R^d)}\d t 
        & \leq \left( \int_0^T \norm{\rho_\tau (t)}^2_{\dot{H}^{-s}(\R^d)} \d t \right)^{1-s}  \left( \int_0^T \norm{\rho_\tau (t)}^2_{\dot{H}^{1-s} (\R^d)} \d t \right)^s \\
        & \leq \left(2 T \calF (\rho_0, \eta_0) \right)^{1-s} \left( \calH (\rho_0, \eta_0) + \left( 1 + T \calF (\rho_0, \eta_0) + \mathrm{m}_2\left[\rho_0\right]+\mathrm{m}_2\left[\eta_0\right]\right) \right)^s,
    \end{align*}
    and by the lower semi-continuity of the norm, we deduce
    \begin{equation}
        \label{eq:conv_H_1-s}
        \rho_\tau, \rho \in L^2((0,T); H^{1-s} (\R^d)).
    \end{equation}
    Consider $\psi \in \calS (\R^d)$. By Proposition \ref{prop:inequalities} and inequality \eqref{eq:proof_energy_1}, we estimate 
    \[
        \norm{ \psi \rho_\tau (t)}^2_{H^{-s}(\R^d)} \leq \norm{\rho_\tau (t)}^2_{H^{-s}(\R^d)} \leq \norm{\rho_\tau (t)}^2_{\dot{H}^{-s}(\R^d)} \leq 2 \calF (\rho_0, \eta_0),
    \]
    obtaining that the family $\{ \psi \rho_\tau (t) \}_\tau$ is bounded in $H^{-s}(\R^d)$. By Proposition \ref{prop:inequalities}, it follows that that $\{ \psi \rho_\tau (t) \}_\tau$ is compact in $H^{-s-\varepsilon}(\R^d)$, for $\varepsilon >0$ small. Since, $\rho_\tau (t) \to \rho (t)$ narrowly, we can find a subsequence $\tau_n$ such that $\psi \rho_{\tau_n} (t) \to \psi \rho (t)$ strongly in $H^{-s-\varepsilon}(\R^d)$ for all $t>0$.
    By Proposition \ref{prop:inequalities}, there exists a constant $c$ such that, for all $t>0$,
    \[
        \norm{\psi \rho_\tau (t) - \psi \rho(t)}^2_{H^{-s-\varepsilon}(\R^d)} \leq c \left( \norm{\rho_\tau (t)}^2_{H^{-s}(\R^d)} + \norm{\rho (t)}^2_{H^{-s}(\R^d)} \right) \leq 4 c \calF (\rho_0, \eta_0),
    \]
    i.e., the $H^{-s-\varepsilon}$-norm is dominated. By the Dominated Convergence Theorem, we deduce
    \[
        \int_0^T \norm{\psi \rho_{\tau_n} (t) - \psi \rho(t)}^2_{H^{-s-\varepsilon}(\R^d)} \d t \rightarrow 0 \quad \mbox{as $n \to \infty$.}
    \]
    By the interpolation inequality, we infer
    \begin{equation}
        \label{eq:estim_H_k}
        \begin{aligned}
            \int_0^T & \norm{\psi \rho_\tau (t)- \psi \rho (t)}^2_{H^k(\R^d)} \d t \\
            & \leq \left( \int_0^T \norm{\psi \rho_\tau (t)- \psi \rho (t)}^2_{H^{-s-\varepsilon}(\R^d)} \d t \right)^{1-\theta} \left( \int_0^T \norm{\psi \rho_\tau (t)- \psi \rho (t)}^2_{H^{1-s}(\R^d)} \d t \right)^\theta,
        \end{aligned}
    \end{equation}
    with $k=(1+\theta) (-s-\varepsilon) + \theta(1-s)$, namely $\theta = (k+s+\varepsilon) / (1+\varepsilon)$. We notice that $0 < \theta < 1$ if and only if $k < 1-s$.
    Now, since
    \[
        \int_0^T \norm{\psi \rho_\tau (t) - \psi \rho (t)}^2_{H^{1-s}(\R^d)} \d t \leq C \int_0^T \norm{\rho_\tau (t) - \rho (t)}^2_{H^{1-s}(\R)} \d t, 
    \]
    due to \eqref{eq:conv_H_1-s}, we get that the left-hand side in \eqref{eq:estim_H_k} vanishes, i.e. 
    \begin{equation}
        \label{eq:conv_H_r}
        \psi \rho_{\tau_n} \to \psi \rho \quad \mbox{strongly in $L^2 ((0,T), H^k (\R^d))$ \ as \ $n \to + \infty$, \ for $k<1-s$ \ and $\psi \in \calS (\R^d)$}.
    \end{equation}
    Let $K\subset \R^d$ and $\varphi \in C_c^\infty (\R^d; \R)$ with $\varphi = 1$ on $K$. Taking $k=0$, we deduce that
    \[
        \norm{\rho_\tau (t) - \rho (t)}^2_{L^2 (K)} \leq \norm{\varphi \rho_\tau (t) - \varphi \rho (t)}^2_{L^2 (\R^d)},
    \]
    and due to \eqref{eq:conv_H_r} we get the first convergence result. 

    \noindent \underline{Statement (ii).} We now deal with the convergence of the convolution terms. Since $\widehat{K_s \ast \rho_\tau (t)}(\xi) =\abs{\xi}^{-2s} \widehat{\rho_\tau (t)} (\xi)$, we get
    \begin{equation} 
        \label{eq:estim_vel}
        \begin{aligned}
            \norm{\nabla K_s \ast \rho_\tau (t) + \nabla K_q \ast \eta_\tau (t)}_{L^2 (\R)} & \leq \norm{\nabla K_s \ast \rho_\tau (t)}_{L^2 (\R)} + \norm{\nabla K_q \ast \eta_\tau (t)}_{L^2 (\R)} \\
            & = \norm{\rho_\tau (t)}_{\dot{H}^{1-2s} (\R^d)} + \norm{\eta_\tau (t)}_{\dot{H}^{1-2q}(\R^d)}.
        \end{aligned}
    \end{equation}
    By the interpolation
    \[
        \norm{\rho_\tau (t)}_{\dot{H}^{1-2s}(\R^d)} \leq \norm{\rho_\tau (t)}^s_{\dot{H}^{-s}(\R^d)} \norm{\rho_\tau (t)}^{1-s}_{\dot{H}^{1-s}(\R^d)},
    \]
    we write
    \begin{align*}
        \int_0^T \norm{\rho_\tau (t)}^2_{\dot{H}^{1-2s} (\R^d)} \d t & \leq \left( \int_0^T \norm{\rho_\tau (t)}^2_{\dot{H}^{s} (\R^d)} \d t \right)^s \left( \int_0^T \norm{\rho_\tau (t)}^2_{\dot{H}^{1-s} (\R^d)} \d t \right)^{1-s} \\
        & \leq \left( 2 T \calF(\rho_0, \eta_0) \right)^s \left( \calH (\rho_0, \eta_0) + c \left( 1 + T \calF (\rho_0, \eta_0) + \mathrm{m}_2\left[\rho_0\right]+\mathrm{m}_2\left[\eta_0\right]\right) \right)^{1-s},
    \end{align*}
    thus the first term in the right hand side of \eqref{eq:estim_vel} is bounded. 
    In order to estimate the second one, we need to use the assumption, and thus distinguish the cases $s<r$ and $r<s$. If $r<s$ we deduce $-r < 1-2q < 1-s < 1-r$.
    By interpolation we obtain
    \begin{equation}
        \label{eq:est_cross_H_1-2q}
        \norm{\eta_\tau (t)}_{\dot{H}^{1-2q}(\R^d)} \leq \norm{\eta_\tau (t)}^{1-\theta}_{\dot{H}^{-r}(\R^d)} \norm{\eta_\tau (t)}^\theta_{\dot{H}^{1-r} (\R^d)},
    \end{equation}
    with $\theta = r+1-2q$. Notice that $0 < \theta < 1$. Applying the Holder's inequality with powers $\frac{1}{1-\theta}$ and $\frac{1}{\theta}$, we arrive at
    \begin{align*}
        \int_0^T \norm{\eta_\tau (t)}^2_{\dot{H}^{1-2q}(\R^d)}\d t & \leq  \left( \int_0^T \norm{\eta_\tau (t)}^2_{\dot{H}^{-r}(\R^d)}\d t \right)^{1-\theta} \left( \int_0^T \norm{\eta_\tau (t)}^2_{\dot{H}^{1-r}(\R^d)}\d t \right)^\theta \\
        & \leq \left( 2 T \calF (\rho_0, \eta_0) \right)^{1-\theta} \left( \calH (\rho_0, \eta_0) + c \left( 1 + T \calF (\rho_0, \eta_0) + \mathrm{m}_2\left[\rho_0\right]+\mathrm{m}_2\left[\eta_0\right]\right) \right)^\theta.
    \end{align*}
    It follows that $\nabla K_s \ast \rho_\tau (t) + \nabla K_q \ast \eta_\tau (t)$ is weakly compact in $L^2 ((0,T); L^2 (\R^d))$. The case $s<r$ is similar. In order to identify the limit, we take $\varphi \in C_c^\infty (\R^d)$. By the Plancherel formula
    \[
        (2\pi)^d \intd \left( \nabla K_s \ast \rho_{\tau_n} + \nabla K_q \ast \eta_{\tau_n} \right) \varphi \d x = - \intd \left( \abs{\xi}^{-2s} \hat{\rho}_{\tau_n} (-\xi) + \abs{\xi}^{-2q} \hat{\eta}_{\tau_n} (-\xi) \right) \xi \cdot \hat{\varphi} (\xi) \d \xi.
    \]
    Integrating in time as $t \in (0,T)$, we deduce
    \begin{align*}
        (2\pi)^d \int_0^T & \intd \left( \nabla K_s \ast \rho_{\tau_n} + \nabla K_q \ast \eta_{\tau_n} \right) \varphi \d x \d t  \\
        & = - i \int_0^T \intd \left( \abs{\xi}^{-2s} \widehat{\rho_{\tau_n} (t)} (-\xi) + \abs{\xi}^{-2q} \widehat{\eta_{\tau_n} (t)} (-\xi) \right) \xi \cdot \hat{\varphi} (t, \xi) \d \xi \d t.
    \end{align*}
    By Dominated Convergence Theorem, since the integrand in the right-hand side is bounded, namely
    \[
        \left\lvert \left( \abs{\xi}^{-2s} \widehat{\rho_{\tau_n} (t)} (-\xi) + \abs{\xi}^{-2q} \widehat{\eta_{\tau_n} (t)} (-\xi) \right) \xi \cdot \hat{\varphi} (t, \xi) \right\rvert \leq \left\lvert \abs{\xi}^{1-2s} + \abs{\xi}^{1-2q} \right\rvert \abs{\hat{\varphi} (\xi) }
    \]
    due to $\rho_{\tau_n}, \eta_{\tau_n} \in \calP(\R^d)$ and $\hat{\varphi} \in \calS (\R^d)$, by the narrow convergence of $(\rho_{\tau_n}, \eta_{\tau_n})$ stated in Theorem \ref{thm:conv_sol}, said term converges to
    \[
        - i \int_0^T \intd \left( \abs{\xi}^{-2s} \widehat{\rho (t)} (-\xi) + \abs{\xi}^{-2q} \widehat{\eta (t)} (-\xi) \right) \xi \cdot \hat{\varphi} (t, \xi) \d \xi \d t,
    \]
    i.e.,
    \[
        \int_0^T \intd \left( \nabla K_s \ast \rho_{\tau_n} + \nabla K_q \ast \eta_{\tau_n} \right) \varphi \d x \d t \rightarrow \int_0^T  \intd \left( \nabla K_s \ast \rho + \nabla K_q \ast \eta \right) \varphi \d x \d t.
    \]
    This proves the weak convergence of the convolutions terms, as stated.
\end{proof}

\begin{lemma}
    \label{lemma:trasport}
    Let $(\rho_0, \eta_0) \in X_2^{s,q}\times X_2^{r,q}$, and let $\{(\rho_\tau^n, \eta_\tau^n)\}_{n \in \N}$ be the minimising sequence constructed in Theorem \ref{thm:conv_sol}. Denote by $T_\rho$ and $T_\eta$ the optimal transportation maps between $\rho_\tau^{n+1}$ and $\rho_\tau^n$ and between $\eta_\tau^{n+1}$ and $\eta_\tau^n$ respectively. If $q$ satisfies
    \[
        \max \left\{ \frac{s}{2}, \frac{r}{2} \right\} \leq q \leq \min \left\{ \frac{s+1}{2}, \frac{r+1}{2} \right\},
    \]
    then
    \begin{align*}
        \intd \nabla \left( K_s \ast \rho_\tau^{n+1}(x) + K_q \ast \eta_\tau^{n+1}(x) \right) \cdot \varphi (x) \rho_\tau^{n+1} (x) \d x =  \frac{1}{\tau} \intd ( T_\rho - \mathrm{id}) \cdot \varphi \rho_\tau^{n+1} \d x, \\
        \intd \nabla \left( K_r \ast \eta_\tau^{n+1}(x) + K_q \ast \rho_\tau^{n+1}(x) \right) \cdot \chi (x) \eta_\tau^{n+1} (x) \d x =  \frac{1}{\tau} \intd ( T_\eta - \mathrm{id}) \cdot \chi \eta_\tau^{n+1} \d x,
    \end{align*}
    for all $\varphi, \chi \in C_c ^\infty (\R^d \times \R^d)$, where $\mathrm{id}$ denotes the identity map. Furthermore, 
    \begin{align*}
        \intd \abs{\nabla \left( K_s \ast \rho_\tau^{n+1}(x) + K_q \ast \eta_\tau^{n+1}(x) \right)}^2 \rho_\tau^{n+1} \d x = \frac{1}{\tau^2} W_2^2 (\rho_\tau^{n+1}, \rho_\tau^n), \\
        \intd \abs{\nabla \left( K_r \ast \eta_\tau^{n+1}(x) + K_q \ast \rho_\tau^{n+1}(x) \right)}^2 \eta_\tau^{n+1} \d x = \frac{1}{\tau^2} W_2^2 (\eta_\tau^{n+1}, \eta_\tau^n).
    \end{align*}
\end{lemma}
\begin{proof}
Let $\varphi \in C_c^\infty (\R^d;\R^d)$, and $\delta > 0$. We set $B_\delta (x) \coloneqq x + \delta \varphi (x)$. We consider two consecutive items in the JKO scheme \eqref{eq:jko_scheme}, namely $(\rho_\tau^n, \eta_\tau^n)$ and $(\rho_\tau^{n+1}, \eta_\tau^{n+1})$, we perturb the first component of the latter by considering the pair $((B_\delta)_\# \rho_\tau^{n+1}, \eta_\tau^{n+1})$. To lighten the notation, we set $\rho_\delta \coloneqq (B_\delta)_\# \rho_\tau^{n+1}$, then it holds
\begin{equation}
    \label{eq:push_forw}
    \intd \psi (y) \rho_\delta (y)\d y = \intd \psi (B_\delta (x)) \rho_\tau^{n+1}(x)\d x,
\end{equation}
for all measurable function $\psi$. Due to \eqref{eq:jko_scheme}, we write
\begin{equation}
    \label{eq:comp_1}
    \frac{1}{2\tau} \left[ \calW_2^2 ((\rho_\delta, \eta_\tau^{n+1}), (\rho_\tau^n, \eta_\tau^n)) - \calW_2^2  ((\rho_\tau^{n+1}, \eta_\tau^{n+1}), (\rho_\tau^n, \eta_\tau^n)) \right]  + \calF(\rho_\delta, \eta_\tau^{n+1}) - \calF(\rho_\tau^{n+1}, \eta_\tau^{n+1}) \geq 0.
\end{equation}

We start by estimating  the Wasserstein distance term in \eqref{eq:comp_1}. We have that
\[
    \calW_2^2 ((\rho_\delta, \eta_\tau^{n+1}), (\rho_\tau^n, \eta_\tau^n)) - \calW_2^2  ((\rho_\tau^{n+1}, \eta_\tau^{n+1}), (\rho_\tau^n, \eta_\tau^n)) = W_2^2 (\rho_\delta, \rho_\tau^n) - W_2^2 (\rho_\tau^{n+1}, \rho_\tau^n).
\]
Denoting by $T_\rho$ the optimal transport map between $\rho_\tau^{n+1}$ and $\rho_\tau^n$, we know that
\begin{equation}
    \label{eq:opt_W_2}
    W_2^2(\rho_\tau^n, \rho_\tau^{n+1}) = \intd \abs{x-T_\rho(x)}^2 \rho_\tau^{n+1}
    (x)\d x.
\end{equation}
In order to estimate the $W_2$-distance between $\rho_\delta$ and $\rho_\tau^n$, we notice that $T_\rho\circ (B_\delta)^{-1}$ is a (not necessarily optimal) transport map between $\rho_\delta$ and $\rho_\tau^n$, and we write
\[
    W_2^2 (\rho_\delta, \rho_\tau^n) \leq \intd \abs{x- T_\rho(B_\delta)^{-1}(x)}^2 \rho_\delta (x) \d x.
\]
Now, by using \eqref{eq:push_forw} and considering the expansion up to the first order in $\delta$, we compute
\begin{align*}
    W_2^2 (\rho_\delta, \rho_\tau^n) & \leq \intd \abs{B(x)-T_\rho(x)}^2 \rho_\tau^{n+1}(x) \d x \\
    & = \intd \abs{x + \delta \varphi(x) - T_\rho(x)}^2 \rho_\tau^{n+1}(x)\d x \\ 
    & = \intd \abs{x-T_\rho(x)}^2  \eta_\tau^{n+1}(x)\d x - 2 \delta \intd (x-T_\rho(x)) \varphi (x) \rho_\tau^{n+1}(x)\d x + o(\delta).
\end{align*}
Therefore
\[
    W_2^2 (\rho_\delta, \rho_\tau^n) - W_2^2 (\rho_\tau^{n+1}, \rho_\tau^n) = - 2 \delta \intd (x-T_\rho(x)) \varphi (x) \rho_\tau^{n+1}(x)\d x + o(\delta).
\]
Dividing by $\delta$ and performing the same computation with $-\delta$ in place of $\delta$, we end up with
\[
    \frac{1}{2\tau} \left[ \calW_2^2 ((\rho_\delta, \eta_\tau^{n+1}), (\rho_\tau^n, \eta_\tau^n)) - \calW_2^2  ((\rho_\tau^{n+1}, \eta_\tau^{n+1}), (\rho_\tau^n, \eta_\tau^n)) \right] = -\frac{1}{\tau} \intd (\mathrm{id} - T_\rho) \cdot \varphi \rho_\tau^{n+1} \d x.
\]
We start considering the terms involving the functional difference in \eqref{eq:comp_1}.
We notice that
\[ 
    \calF(\rho_\delta, \eta_\tau^{n+1}) - \calF(\rho_\tau^n, \eta_\tau^n) = \frac{1}{2} \left( \norm{\rho_\delta}^2_{\dot{H}^{-s}} - \norm{\rho_\tau^{n+1}}^2_{\dot{H}^{-s}} \right) + \frac{1}{\delta} \left( \langle \rho_\delta, \eta_\tau^{n+1} \rangle_{\dot{H}^{-q}(\R^d)} - \langle \rho_\tau^{n+1}, \eta_\tau^{n+1} \rangle_{\dot{H}^{-q}(\R^d)} \right).
\]
In order to estimate the self-interaction term, involving the $\dot{H}^{-s}$-norms, we follow \cite[Lemma 5.1]{lis_main_seg}. Recall that for all $a, b \in \mathbb{C}$, it holds $\abs{a}^2 - \abs{b}^2 = (\bar{a} + \bar{b})(a-b)+\bar{a}b-\bar{b}a$, and since 
\[
    \intd \abs{\xi}^{-2s} \hat{\rho}_\delta (-\xi) \hat{\rho}_\tau^{n+1} (\xi) \d \xi = \intd \abs{\xi}^{-2s} \hat{\rho}_\delta (\xi) \hat{\rho}_\tau^{n+1} (-\xi) \d \xi,
\]
we have
\[
    \frac{1}{2} \left( \norm{\rho_\delta}^2_{\dot{H}^{-s}} - \norm{\rho_\tau^{n+1}}^2_{\dot{H}^{-s}} \right) = \frac{1}{2} \frac{1}{(2 \pi)^d} \intd \abs{\xi}^{-2s} \left( \hat{\rho}_\delta (-\xi) + \hat{\rho}_\tau^{n+1} (-\xi) \right) \left( (\hat{\rho}_\delta (\xi) - \hat{\rho}_\tau^{n+1} (\xi) \right) \d \xi,
\]
and thus
\begin{equation}
    \label{eq:lim_self_term}
    (2\pi)^d \left( \norm{\rho_\delta}^2_{\dot{H}^{-s}} - \norm{\rho_\tau^{n+1}}^2_{\dot{H}^{-s}} \right) = \intd \abs{\xi}^{1-2s} \left( \hat{\rho}_\delta (-\xi) + \hat{\rho}_\tau^{n+1} (-\xi) \right) \abs{\xi}^{-1} \left( \hat{\rho}_\delta (\xi) - \hat{\rho}_\tau^{n+1} (\xi) \right) \d \xi.
\end{equation}
The first step in this direction is to prove that $\abs{\xi}^{1-2s} \hat{\rho}_\delta (-\xi)$ converges to $\abs{\xi}^{1-2s} \hat{\rho}_\tau^{n+1}(-\xi)$ strongly in $L^2 (\R^d)$ as $\delta \to 0$. In oreder to do it, we notice that $B_\delta$ is a global diffeomorphism, and in particular there is a $\delta_1 >0$ such that
\[
    \frac{1}{2} \leq \det (\nabla B_\delta (x)) \leq \frac{3}{2}, \qquad \mbox{for all $x \in \R^d$,} \quad \mbox{for $\delta \in [0, \delta_1]$},
\]
due to the fact that $\varphi \in C_c^\infty (\R^d ; \R^d)$. We have that also $B_\delta^{-1}$ is a global diffeomorphism, and $\abs{B_\delta (x) - B_\delta (y)} \geq c \abs{x-y}$ for all $\delta \in [0, \delta_1]$ and $x,y \in \R^d$, for some $c >0$. Now, since
\[
    \rho_\delta = h_\delta \rho_\tau^{n+1} \circ B_\delta ^{-1} + \rho_\tau^{n+1} \circ B_\delta^{-1},
\]
with $h_\delta = \det \nabla B_\delta^{-1} - 1$, due to \cite[Corollary 1.60, Theorem 1.62]{bahouri}, we have
\[
    \norm{\rho_\tau^{n+1} \circ B_\delta^{-1}}_{H^{1-s}(\R^d)} \leq c \norm{\rho_\tau^{n+1}}_{H^{1-s}(\R^d)}, \quad \mbox{\and} \quad \norm{h_\delta (\rho_\tau^{n+1} \circ B_\delta^{-1})}_{H^{1-s}(\R^d)} \leq c \norm{\rho_\tau^{n+1}}_{H^{1-s}(\R^d)},
\]
for all $\delta \in [0, \delta_1]$. Combining the two results above, we have
\[
    \norm{\rho_\delta - \rho_\tau^{n+1}}_{H^{1-s} (\R^d)} \leq c + \norm{\rho_\tau^{n+1}}_{H^{1-s}(\R^d)},
\]
for all $\delta \in [0, \delta_1]$. Since the support of $\rho_\delta - \rho_\tau^{n+1}$ is compact, i.e, $\mathrm{supp} (\rho_\delta - \rho_\tau^{n+1})= \mathrm{supp} \varphi$, and $\rho_\delta$ converges to $\rho_\tau^{n+1}$ narrowly when $\delta$ vanishes, then by Rellich-Kondrachov Theorem and the narrow convergence in Theorem \ref{thm:conv_sol}, we deduce
\[
    \norm{\rho_\delta - \rho_\tau^{n+1}}_{H^k(\R^d)}\to 0 \quad \mbox{strongly as $\delta \to 0$ for all $k < 1-s$.}
\]
By \eqref{eq:interp_H_dot} it holds that
\begin{align*}
    \norm{\abs{\xi}^{1-2s} \left(\hat{\rho}_\delta - \hat{\rho}_\tau^{n+1}\right)}_{L^2(\R^d)} & = \norm{\hat{\rho}_\delta - \hat{\rho}_\tau^{n+1}}_{\dot{H}^{1-2s} (\R^d)} \\
    & \leq \norm{\hat{\rho}_\delta - \hat{\rho}_\tau^{n+1}}^{1-\theta}_{\dot{H}^{-s} (\R^d)} \norm{\hat{\rho}_\delta - \hat{\rho}_\tau^{n+1}}^\theta_{\dot{H}^k (\R^d)},
\end{align*}
with $1-2s = (1-\theta) (-s) + \theta k$.  For $k \in (\max\{ 1-2s, 0 \},1-s)$, we obtain
\begin{equation}
    \label{eq:strong_self}
    \abs{\xi}^{1-2s} \hat{\rho}_\delta (-\xi) \rightarrow \abs{\xi}^{1-2s} \hat{\rho}_\tau^{n+1}(-\xi) \quad \mbox{strongly in $L^2 (\R^d)$ as $\delta\to 0$.}
\end{equation}
We now show that $\abs{\xi}^{-1} \frac{1}{\delta} (\hat{\rho}_\delta (\xi) - \hat{\rho}_\tau^{n+1}(\xi))$ converges to $-i \abs{\xi}^{-1} \xi \cdot (\widehat{\varphi \rho}_\tau^{n+1}) (\xi)$ weakly in $L^2(\R^d)$.
Defining the function $g_\xi (\delta) : [0, +\infty) \to \R$ as $g_\xi (\delta) \coloneqq \hat{\rho}_\delta (\xi)$, i.e.,
\[
    g_\xi (\delta)=\hat{\rho}_\delta (\xi)= \intd e^{-i \xi \cdot (x+\delta \varphi (x)} \varphi (x) \rho_\tau^{n+1}(x)\d x.
\]
we have that $g_\xi$ is $C^1$ in $\delta$ and
\[
    g'_\xi (\delta) = - i\xi \cdot \intd e^{-i \xi \cdot (x+\delta \varphi(x)  )} \varphi (x) \rho_\tau^{n+1}(x)\d x,
\]
where we have used the Dominated Convergence Theorem.
By the mean value Theorem, for all $\xi$ and $\delta>0$ there exists a $\delta_\xi \in [0, \delta)$ such that
\[
    \frac{1}{\delta} ( \hat{\rho}_\delta (\xi) - \hat{\rho}_\tau^{n+1} (\xi) )= g'_\xi (\delta_\xi) = - i \xi \cdot \widehat{(B_{\delta_\xi}) _\# (\varphi \rho_\tau^{n+1}}) (\xi),
\]
by using the definition of $B_\delta$. Since $g'_\xi (\delta_\xi)$ is dominated, namely $\abs{ g'_\xi (\delta_\xi)} \leq \abs{\xi} \norm{\varphi}_{L^\infty(\R^d)} \norm{\rho_\tau^{n+1}}_{L^1(\R^d)}$, then
\[
    \abs{\xi}^{-1} \frac{1}{\delta} \left( \hat{\rho}_\delta (\xi) - \hat{\rho}_\tau^{n+1} (\xi) \right) \rightarrow -i \abs{\xi}^{-1} \xi \cdot (\widehat{\varphi \rho}_\tau^{n+1}) (\xi)
\]
in the sense of distributions. Moreover, it holds that $\norm{(B_\delta)_\# (\varphi \rho_\tau^{n+1})}_{L^2(\R^d)} \leq \norm{\varphi \rho_\tau^{n+1}}_{L^2(\R^d)}$, therefore, $ \abs{\xi}^{-1} \frac{1}{\delta} ( \hat{\rho}_\delta (\xi) - \hat{\rho}_\tau^{n+1} (\xi) )$ is uniformly bounded in $L^2 (\R^d)$. As a consequence,
\begin{equation}
    \label{eq:weak_self}
    \abs{\xi}^{-1} \frac{1}{\delta} (\hat{\rho}_\delta (\xi) - \hat{\rho}_\tau^{n+1}(\xi) \rightharpoonup -i \abs{\xi}^{-1} \xi \cdot (\widehat{\varphi \rho}_\tau^{n+1}) (\xi) \quad \mbox{weakly in $L^2(\R^d)$.}
\end{equation}
Dividing by $\delta$ in \eqref{eq:lim_self_term} and passing to the limit as $\delta\to0$, using the pairing of strong convergence and weak convergence in \eqref{eq:strong_self} and \eqref{eq:weak_self}, we compute
\begin{align*}
    (2\pi)^d \lim_{\delta \to 0} \frac{1}{\delta} \left( \norm{\rho_\delta}^2_{\dot{H}^{-s}} - \norm{\rho_\tau^{n+1}}^2_{\dot{H}^{-s}} \right) 
   & = - i \intd \abs{\xi}^{-2s} \hat{\rho}_\tau^{n+1} (-\xi) \xi \cdot (\widehat{\varphi \rho}_\tau^{n+1}) (\xi) \d \xi \\
    & = -i \sum_{j=1}^d \intd \abs{\xi}^{-2s} \rho_\tau^{n+1}(-\xi) \xi_j \cdot (\widehat{\varphi_j \rho}_\tau^{n+1}) (\xi) \d \xi \\
    & = (2 \pi)^d \sum_{j=1}^d \intd \de_{x_j} \left( \abs{x}^{-2s} \rho_\tau^{n+1}(x) \right) \varphi_j (x) \rho_\tau^{n+1} (x) \d x \\
    & = (2 \pi)^d \intd \nabla \left( \abs{x}^{-2s} \rho_\tau^{n+1}(x) \right) \cdot \varphi (x) \rho_\tau^{n+1} (x) \d x \\
    & = (2 \pi)^d \intd \nabla \left( K_s \ast \rho_\tau^{n+1}(x) \right) \cdot \varphi (x) \rho_\tau^{n+1} (x) \d x,
\end{align*}
where we used the Plancherel Theorem and the relation $\widehat{K_s \ast \rho}_\tau^{n+1} (\xi) =\abs{\xi}^{-2s} \hat{\rho}_\tau^{n+1} (\xi)$.

We are now left to estimate the cross-interaction term in \eqref{eq:comp_1} involving the inner product in $\dot{H}^{-q}$.  We have that
\[
    \left\langle \rho_\delta - \rho_\tau^{n+1}, \eta_\tau^{n+1} \right\rangle_{\dot{H}^{-q}(\R^d)}  = \intd \abs{\xi}^{1-2q} \hat{\eta}_\tau^{n+1}(-\xi)\abs{\xi}^{-1} \left( \hat{\rho}_\delta (\xi) - \hat{\rho}_\tau^{n+1} (\xi) \right) \d \xi.
\]
By noticing that
\[
    \norm{\abs{\xi}^{1-2q} \hat{\eta}_\tau^{n+1}}_{L^2(\R^d)}=\norm{ \eta_\tau^{n+1}}_{\dot{H}^{1-2q} (\R^d)},
\]
using the same argument of the proof of Lemma \ref{lemma:convergence} and estimate \eqref{eq:est_cross_H_1-2q}, the assumption on $q$ ensures that $\abs{\xi}^{1-2q} \hat{\eta}_\tau^{n+1} \in L^2 (\R^d)$. Due to \eqref{eq:weak_self}, performing a similar computation as above, we get
\begin{align*}
    (2\pi)^d \lim_{\delta \to 0} \frac{1}{\delta} \left\langle \rho_\delta - \rho_\tau^{n+1}, \eta_\tau^{n+1} \right\rangle_{\dot{H}^{-q}(\R^d)} 
    & = - i \intd \abs{\xi}^{-2q} \hat{\eta}_\tau^{n+1} (-\xi) \xi \cdot (\widehat{\varphi \rho}_\tau^{n+1}) (\xi) \d \xi \\
    & = (2 \pi)^d \sum_{j=1}^d \intd \de_{x_j} \left( \abs{x}^{-2q} \eta_\tau^{n+1}(x) \right) \varphi_j (x) \rho_\tau^{n+1} (x) \d x \\
    & = (2 \pi)^d \intd \nabla \left( K_q \ast \eta_\tau^{n+1}(x) \right) \cdot \varphi (x) \rho_\tau^{n+1} (x) \d x.
\end{align*}
Combining together all the estimates, we conclude that
\begin{equation}
    \label{eq:est_for_sol}
    \intd \nabla \left( K_s \ast \rho_\tau^{n+1}(x) + K_q \ast \eta_\tau^{n+1}(x) \right) \cdot \varphi (x) \rho_\tau^{n+1} (x) \d x =  -\frac{1}{\tau} \intd (\mathrm{id} - T_\rho) \cdot \varphi \rho_\tau^{n+1} \d x,
\end{equation}
for all $\varphi \in C_c^\infty(\R^d ; \R^d)$. A consequence is that the two integrand functions are equal almost everywhere, namely
\[
    \tau \rho_\tau^{n+1} \nabla \left( K_s \ast \rho_\tau^{n+1}(x) + K_q \ast \eta_\tau^{n+1}(x) \right) = (T_\rho-\mathrm{id}) \rho_\tau^{n+1}.
\]
Combining this with \eqref{eq:opt_W_2}, the proof is complete.
\end{proof}
With a similar argument we can prove the following result.
\begin{cor}\label{cor:degiorgiid}
    Under the assumptions of Lemma \ref{lemma:trasport}, for any $t\in(n\tau,(n+1)\tau]$ we have that
        \begin{align*}
        \intd \abs{\nabla \left( K_s \ast \tilde{\rho}_\tau(t,x) + K_q \ast \tilde{\eta}_\tau(t,x) \right)}^2 \tilde{\rho}_\tau(t,x) \d x = \frac{1}{\left(t-n\tau\right)^2} W_2^2 (\tilde{\rho}_\tau(t), \rho_\tau^n), \\
            \intd \abs{\nabla \left( K_r \ast \tilde{\eta}_\tau(t,x) + K_q \ast \tilde{\rho}_\tau(t,x) \right)}^2 \tilde{\eta}_\tau(t,x) \d x = \frac{1}{\left(t-n\tau\right)^2} W_2^2 (\tilde{\eta}_\tau(t), \eta_\tau^n),
    \end{align*}
\end{cor}

We are now in the position to state the discrete energy equality.

\begin{prop}\label{prop:dei}    
Let $T>0$ be fixed and let $(\rho_\tau(t), \eta_\tau(t))$ and $(\tilde{\rho}_\tau(t), \tilde{\eta}_\tau(t))$ be the interpolations defined in \eqref{eq:seq_interpol} and \eqref{e:degiorgi} respectively, for $t\in (0,T]$. 
Then the following discrete energy equality holds
\[
\begin{aligned}
    &\frac{1}{2}\int_0^T\int_{\R^d} \abs{ \nabla \left( K_s\ast \rho_\tau+K_q \ast \eta_\tau \right) }^2\rho_\tau\d x\d t+  \frac{1}{2}\int_0^T\int_{\R^d}\abs{\nabla \left( K_r\ast \eta_\tau +K_q \ast \rho_\tau \right) }^2\eta_\tau\d x\d t\\
    & \quad + \frac{1}{2}\int_0^T\int_{\R^d}\abs{\nabla \left( K_s\ast \tilde{\rho}_\tau +K_q \ast \tilde{\eta}_\tau \right) }^2 \tilde{\rho}_\tau\d x\d t+  \frac{1}{2}\int_0^T\int_{\R^d}\abs{ \nabla \left( K_r\ast \tilde{\eta}_\tau +K_q \ast \tilde{\rho}_\tau \right) }^2\tilde{\eta}_\tau\d x\d t \\
    & \quad +\calF(\rho_\tau(T), \eta_\tau(T))=\calF(\rho_0, \eta_0).
\end{aligned}
\]
\end{prop}
\begin{proof}
Consider the De Giorgi variational interpolation introduced in \eqref{e:degiorgi}. By a classical argument, for $n\in \N$ such that $t\in \left(n\tau,(n+1)\tau\right]$ we have that
\begin{align*}
    \frac{1}{2\tau}\calW_2^2((\rho_\tau^n, \eta_\tau^n), (\rho_\tau^{n+1}, \eta_\tau^{n+1}))&+\frac{1}{2}\int_{n\tau}^{(n+1)\tau}\frac{\calW_2^2 ((\rho_\tau^n, \eta_\tau^n), (\tilde{\rho}_\tau(t), \tilde{\eta}_\tau(t)))}{(t-n\tau)^2}\,\d t\\
    &+\calF(\rho_\tau^{n+1}, \eta_\tau^{n+1})=\calF(\rho_\tau^n, \eta_\tau^n),
\end{align*}
see \cite[Lemma 3.2.2]{ags}. Summing for $n=0,\ldots,N-1$, using Lemma \ref{lemma:trasport} and Corollary \ref{cor:degiorgiid}  we have the thesis.
\end{proof}

We are now ready to prove the our main result stated in Theorem \ref{thm:intro_main}.

\begin{proof}
    Let $\varphi \in C_c^\infty ((0,+\infty), \R^d)$ be a given test function and consider the estimate \eqref{eq:est_for_sol} with 
    $\nabla \varphi$ in place of $\varphi$. We get
    \begin{equation}
        \label{eq:weak_sol_1}
        \int_0^\infty \intd \left( \nabla (K_s \ast \rho_\tau + K_q \ast \eta_\tau \right) \cdot \nabla \varphi \rho_\tau \d x \d t = \frac{1}{\tau} \int_0^\infty \intd (T_\rho-\mathrm{id} ) \cdot \nabla \varphi \rho_\tau \d x \d t,
    \end{equation}
    where $T_\rho (t)$ is the optimal transport map between $\rho_\tau^{n+1}$ and $\rho_\tau^n$ as $t \in (n\tau, (n+1)\tau]$. Concerning the left-hand side of \eqref{eq:weak_sol_1}, by Lemma \ref{lemma:convergence} there is a subsequence $\tau_n$ such that
    \[
        \int_0^\infty \intd  \nabla \left(K_s \ast \rho_\tau + K_q \ast \eta_\tau \right) \cdot \nabla \varphi \rho_\tau \d x \d t \rightarrow \int_0^\infty \intd  \nabla \left(K_s \ast \rho + K_q \ast \eta \right) \cdot \nabla \varphi \rho \d x \d t.
    \]
    By using the definition of optimal transport map $T_\rho$ and the notion of push-forward in \eqref{eq:push_forw}, by considering the Taylor expansion of $\varphi  (T_\rho(t))$ around $x$ and the estimates in the proof of Theorem \ref{thm:conv_sol}, the right-hand side of \eqref{eq:weak_sol_1} can be written as
    \[
        \intd \varphi (x) [\rho_\tau^n (x) - \rho_\tau^{n+1} (x)] \d x = \intd (x- T_\rho(x)) \cdot \nabla \varphi (T_\rho (x)) \rho_\tau^n (x) \d x + o (\tau).
    \]
    Assuming $\rho_\tau^n = \rho_\tau (t)$ as $t \in (n\tau, (n+1)\tau]$, taking $s<t$ and
    \[
        h= \left\lceil \frac{t}{\tau} \right\rceil, \qquad k= \left\lceil \frac{s}{\tau} \right\rceil,
    \]
    we sum as $n=h, \ldots, k-1$ and we get
    \[
        \intd \varphi(x) [\rho_\tau (t,x) - \rho_\tau (s,x)] \d x = \sum_{n=h}^{k-1} \intd (x- T_\rho(x)) \cdot \nabla \varphi (T_\rho (x)) \rho_\tau^n (x) \d x + \sum_{n=h}^{k-1} o (\tau).
    \]
    Diving by $s-t$, sending first $\tau \to 0$ and then $s \to t$, we deduce
    \[
        \intd \varphi \de_t \rho \d x = - \intd \rho \nabla \left( K_s \ast \rho + K_q \ast \eta \right) \cdot \nabla \varphi \d x.
    \]
    Finally, by considering a test function $\psi \in C_c^\infty (0,T)$, and integrating in time, we conclude
    \[
        \int_0^T \intd \varphi \psi \de_t \rho  \d x \d t = - \int_0^T \intd \rho \nabla \left( K_r \ast \eta + K_q \ast \rho \right) \cdot \nabla \varphi \psi \d x \d t,
    \]
    that proves convergence to weak solution. The energy dissipation inequality follows from a standard lower semi-continuity argument. Indeed, by invoking and \cite[Theorem 5.4.4]{ags} and \eqref{eq:proof_energy_1} we have
    \begin{equation*}
        \liminf_{n\to \infty}\int_0^T\int_{\R^d} \abs{ \nabla \left( K_s \ast \rho_{\tau_n} + K_q \ast \eta_{\tau_n} \right) }^2\rho_{\tau_n}\,\d x\geq \int_0^T\int_{\R^d}\abs{ \nabla \left( K_s \ast \rho + K_q \ast \eta\right) }^2\rho\,\d x.
    \end{equation*}
    The same inequality holds for the second species. Note that by a triangulation, \eqref{eq:proof_prop_2} and the definition in \eqref{e:degiorgi} we have
    \[
        \calW_2^2\left((\tilde{\rho}_\tau(t), \tilde{\eta}_\tau(t)),(\rho(t), \eta(t))\right)\leq C\tau \calF((\rho_0, \eta_0)).
    \]
    Thus, $(\tilde{\rho}_{\tau_n}(t), \tilde{\eta}_{\tau_n}(t))\to(\rho(t), \eta(t))$ narrowly as $n\to +\infty$. By a similar argument to the one used above we can argue the weak $L^2$-convergence of $\nabla \left( K_s \ast \tilde{\rho}_{\tau_n} + K_q \ast \tilde{\eta}_{\tau_n} \right)$ to the same limit of $\nabla \left( K_s \ast \rho_{\tau_n} + K_q \ast \eta_{\tau_n} \right)$ and similarly for the second species. Hence, we get
    \begin{align*}
     \calF(\rho(T), \eta(T)) & +\frac{1}{2}\int_0^T\int_{\R^d}\abs{\nabla \left( K_s \ast \rho(t)+K_q\ast \eta (t) \right)}^2\rho\d x\d t \\
     & \quad +  \frac{1}{2}\int_0^T\int_{\R^d}\abs{\nabla \left( K_r \ast \eta(t) + K_q \ast \rho(t) \right) }^2\eta\d x\d t\leq\calF(\rho_0, \eta_0),
    \end{align*}
    and the assertion is proved.
\end{proof}

\section{Conclusion and perspectives}
In this paper, we have investigated the existence of weak solutions for a two-species system of nonlocal continuity equations driven by Riesz potentials. A key aspect of our analysis has been the use of the formal gradient flow structure associated to the interaction energy functional in the Wasserstein product space. The main contribution of this work lies in the treatment of  singular Riesz kernels, particularly for the cross-interaction terms. While the system's coupling prevents a straightforward application of standard single-species gradient flow theory, we have shown that the variational framework provided by the JKO minimizing movement scheme remains a powerful tool for establishing existence. Our results extend the existing literature by allowing for singular cross-interaction kernels under a symmetry assumption, bridging the gap between smooth multi-species models and singular single-species equations. Future research could explore the cases where the system might deviate from a pure gradient flow structure, such as under non-symmetric singular cross-interactions or the small inertia limit in the spirit of \cite{choi_fagioli_iorio}.

\section*{Acknowledgments}
SF is partially supported by the Italian “National Centre for HPC, Big Data and Quantum Computing” - Spoke 5 “Environment and Natural Disasters” and by the Ministry of University and Research (MIUR) of Italy under the grant PRIN 2020- Project N. 20204NT8W4, Nonlinear Evolutions PDEs, fluid
dynamics and transport equations: theoretical foundations and applications.
SF and VI are partially supported by the InterMaths Network, \url{www.intermaths.eu}. SF and VI are also partially supported by the INdAM-GNAMPA project 2025 code CUP E5324001950001 ``Teoria e applicazioni dei modelli evolutivi:
trasporto ottimo, metodi variazionali e
approssimazioni particellari deterministiche'', and by the INdAM-GNAMPA project 2026 code CUP E53C25002010001 ``Modelli di reazione-diffusione-trasporto: dall'analisi alle applicazioni''.

\appendix
\section{Useful well-known results}\label{sec:appendix}
This section makes a list of well-known theorems and results that are useful for this manuscript.
\begin{thm}[A refined version of Ascoli-Arzel\`{a} Theorem \cite{ags}]\label{ascoli}
    Let $(\mathcal{S}, d)$ be a complete metric space and let $T>0.$ Let $K\subset \mathcal{S}$ be a sequentially compact set w.r.t. a weaker topology $\sigma$ on $\mathcal{S},$ and let $u_n:[0,T]\to \mathcal{S}$ be curves such that
    \begin{gather*}
    u_n(t)\in K \quad \forall n\in\N, \quad t\in[0,T], \\
    \limsup_{n\to \infty} d(u_n(s),u_n(t))\le \omega(s,t)\quad \forall s,t\in[0,T], 
    \end{gather*}
    for a (symmetric) function $\omega:[0,T]\times[0,T]\to [0,\infty),$ such that 
    \[
    \lim_{(s,t)\to(r,r)}\omega(s,t)=0 \quad \forall r\in[0,T]\setminus \mathcal{C},
    \]
    where $\mathcal{C}$ is an (at most) countable subset of $[0,T].$ Then there exists an increasing subsequence $k\mapsto n(k)$  and a limit curve $u:[0,T]\to \mathcal{S}$ such that 
    \[
        u_{n(k)}(t)\xrightarrow{\sigma} u(t)\quad \forall t\in[0,T],\quad u \quad \text{is d-continuous in}\quad  [0,T]\setminus \mathcal{C}.
    \]
\end{thm}

\begin{thm}[Extended Aubin-Lions Lemma]\label{aubin}
On a Banach space $(X,d)$, let  $\mathscr{Y}:X\to [0,\infty]$ be a given lower semi-continuous functional with relatively compact sub-levels in $X$. Let $\mathfrak{d}:X\times X\to [0,\infty]$ be a given pseudo-distance on $X$, that is, $\mathfrak{d}$ is lower semi-continuous and $\mathfrak{d}(\rho,\eta)=0$ for any $\rho,\eta\in X$ with $\mathscr{Y}[\rho],\mathscr{Y}[\eta]<\infty$ implies $\rho=\eta$.
    Let further $U$ be a set of measurable functions $u:[0,T]\to X,$ with a fixed $T>0.$ If 
    \[
    \sup_{u\in U}\int_0^T \mathscr{Y}[u(t)]dt<\infty \quad \text{and}\quad \lim_{h\downarrow 0}\sup_{u\in U}\int_0^{T-h} \mathfrak{d}(u(t+h),u(t))dt = 0,
    \]
    then $U$ contains an infinite sequence $\{u_n\}_{n\in\N}$ that converges in measure (w.r.t.\ $t\in[0,T]$) to a limit $u:[0,T]\to X.$

\end{thm}

\begin{definition}[$k$-flow]
    A semigroup $S^\Psi:[0,+\infty]\times \calP_2(\R^d)\to \calP_2(\R^d)$ is a \emph{$k$-flow} for a functional $\Psi:\calP_2(\R^d)\to \R\cup\{+\infty\}$ with respect to the $2$-Wasserstein distance $W_2$ if, for any arbitrary $\rho\in\calP_2(\R^d),$ the curve $t\mapsto S^\Psi_t\rho$ is absolutely continuous on $[0,+\infty]$ and satisfies the evolution variational inequality (E.V.I.)
    \begin{equation} \label{eq:evi}
        \frac{1}{2}\frac{d^+}{d t} W_2^2(S^\Psi_t\rho,\Tilde{\rho})+\frac{k}{2} W_2^2(S^\Psi_t \rho,\Tilde{\rho}) \leq \Psi(\Tilde{\rho})-\Psi(S^\Psi_t\rho),
    \end{equation}
    for all $t>0,$ and for any $\Tilde{\rho}\in\calP_2(\R^d)$ with $\Psi(\Tilde{\rho})<\infty.$
\end{definition}
The symbol $\frac{d^+}{d t}$ stands for the limit superior of the respective difference quotients and equals the derivative if the latter exists.

\begin{thm}
    Assume that a functional $\Psi:\calP_2(\R^d)\to \R\cup\{+\infty\}$ is $\lambda$-convex (along geodesics), with a modulus of convexity $\lambda\in \R,$ that is, along every constant speed geodesic $\rho:[0,1]\to \calP_2(\R^d),$
    \[
    \Psi[\rho(t)]\le (1-t)\Psi[\rho(0)]+t\Psi[\rho(1)]-\frac{\lambda}{2}t(1-t)W_2^2(\rho(0),\rho(1))
    \]
    holds for every $t\in[0,1].$ Then $\Psi$ posses a uniquely determined $k$-flow, with some $k\le \lambda.$ Conversely, if a functional $\Psi$ possesses a $k$-flow, and if it is monotonically non-increasing along that flow, then $\Psi$ is $\lambda$-convex, with some $\lambda\ge k.$
\end{thm}
Below we state the flow interchange lemma, see \cite{matthes2009family} for more details.

\begin{lemma}(Flow interchange)\label{flowinterlem}
    Let $\Psi:\calP_2(\R^d)\to (-\infty,\infty]$ be a lower semi-continuous functional which possesses a $k$-flow $S^\Psi.$ Define further the dissipation of $\calF$ along $S^\Psi$ by 
    \[
        \mathsf{D}^\Psi\calF(\rho)\coloneqq \limsup_{s\downarrow0}\frac{1}{s}\left(\calF(\rho)-\calF(S^\Psi_s\rho)\right)
    \]
    for every $\rho\in\calP_2(\R^d).$ If $\rho_\tau^{n-1}$ and $\rho_\tau^n$ are two consecutive steps of the minimizing movement scheme \eqref{eq:jko_scheme} then 
    \[
        \Psi(\rho_\tau^{n-1})-\Psi(\rho_\tau^n)\ge \tau \mathsf{D}^\Psi\calF(\rho_\tau^n)+\frac{k}{2}W_2^2(\rho_\tau^n,\rho_\tau^{n-1}),
    \]
    In particular, if $\Psi(\rho_\tau^{n-1})<\infty$, then $\mathsf{D}^\Psi\calF(\rho_\tau^n)<\infty$.
\end{lemma}
\begin{cor}\label{flowintercor}
    Under the assumptions of Lemma \ref{flowinterlem}, let $k$-flow $S^\Psi$ be such that for every $n\in\N$, the curves $t\mapsto S^\Psi_t\rho_\tau^n$ lies in $L^\gamma(\R^d)$, where it is differentiable for every $t>0$ and continuous at $t=0$. Moreover, let $\mathfrak{R}:\calP_2(\R^d)\to (-\infty,\infty]$ satisfy
    \begin{equation*}
        \liminf_{s\downarrow 0}\left(-\frac{d}{dt} \calF(S^\Psi_t \rho_\tau^n)\Big|_{t=s}\right)\ge \mathfrak{R}(\rho_\tau^n).
    \end{equation*}
    Then the following two estimates hold:
    \begin{align*}
        & \Psi(\rho_\tau^{n-1})-\Psi(\rho_\tau^n)\ge \tau\mathfrak{R}(\rho_\tau^n)+\frac{k}{2}W_2^2(\rho_\tau^n,\rho_\tau^{n-1})\quad \text{for every }\ n\in \N, \\
        & \Psi(\rho_\tau^N)\le \Psi(\rho_0)-\tau\sum_{n=1}^N \mathfrak{R}(\rho_\tau^n)+\tau\max(0,-k)\calF(\rho_0)\quad \text{for every }\ N\in\N.
    \end{align*}
\end{cor}




\begin{thebibliography}{10}

\bibitem{adams1999function}
D.~R. Adams and L.~I. Hedberg.
\newblock {\em Function Spaces and Potential Theory}.
\newblock Grundlehren der mathematischen Wissenschaften. Springer Berlin
  Heidelberg, 1999.

\bibitem{ags}
L.~Ambrosio, N.~Gigli, and G.~Savare.
\newblock {\em Gradient Flows: In Metric Spaces and in the Space of Probability
  Measures}.
\newblock Lectures in Mathematics. ETH Z{\"u}rich. Birkh{\"a}user Basel, 2005.

\bibitem{bahouri}
H.~Bahouri, J.~Y. Chemin, and R.~Danchin.
\newblock {\em {F}ourier Analysis and Nonlinear Partial Differential
  Equations}.
\newblock Grundlehren der mathematischen Wissenschaften. Springer Berlin
  Heidelberg, 2011.

\bibitem{bertozzi1}
A.~L. Bertozzi and J.~Brandman.
\newblock Finite-time blow-up of {$L^\infty$}-weak solutions of an aggregation
  equation.
\newblock {\em Commun. Math. Sci.}, 8(1):45--65, 2010.

\bibitem{bertozzi2}
A.~L. Bertozzi, J.~A. Carrillo, and T.~Laurent.
\newblock Blow-up in multidimensional aggregation equations with mildly
  singular interaction kernels.
\newblock {\em Nonlinearity}, 22(3):683--710, 2009.

\bibitem{bertozzi3}
A.~L. Bertozzi and T.~Laurent.
\newblock The behavior of solutions of multidimensional aggregation equations
  with mildly singular interaction kernels.
\newblock {\em Chin. Ann. Math. Ser. B}, 30(5):463--482, 2009.

\bibitem{bertozzi4}
A.~L. Bertozzi, T.~Laurent, and J.~Rosado.
\newblock {$L^p$} theory for the multidimensional aggregation equation.
\newblock {\em Comm. Pure Appl. Math.}, 64(1):45--83, 2011.

\bibitem{manev_1}
A.~V. Bobylev, P.~Dukes, R.~Illner, and H.~D. Victory.
\newblock On {V}lasov–{M}anev equations. I: Foundations, properties, and
  nonglobal existence.
\newblock {\em Journal of Statistical Physics}, 88(3):885--911, 1997.

\bibitem{manev_2}
A.~V. Bobylev, P.~Dukes, R.~Illner, and H.~D.~and Victory.
\newblock On {V}lasov–{M}anev equations, II: Local existence and uniqueness.
\newblock {\em Journal of Statistical Physics}, 91(3):625--654, 1998.

\bibitem{boi}
S.~Boi, V.~Capasso, and D.~Morale.
\newblock Modeling the aggregative behavior of ants of the species {\it
  {p}olyergus rufescens}.
\newblock {\em Nonlinear Anal. Real World Appl.}, 1(1):163--176, 2000.
\newblock Spatial heterogeneity in ecological models (Alcal{\'a} de Henares,
  1998).

\bibitem{caff_serf_vazq_reg_sol_pm}
L.~Caffarelli, F.~Soria, and J.~L. Vázquez.
\newblock Regularity of solutions of the fractional porous medium flow.
\newblock {\em Journal of the European Mathematical Society},
  015(5):1701--1746, 2013.

\bibitem{caff_vazq_nonlinear_pm}
L.~{Caffarelli} and J.~L. {Vazquez}.
\newblock Nonlinear porous medium flow with fractional potential pressure.
\newblock {\em Archive for Rational Mechanics and Analysis}, 202(2):537--565,
  November 2011.

\bibitem{CaCho21}
J.~A. Carrillo and Y.-P. Choi.
\newblock Mean-field limits: From particle descriptions to macroscopic
  equations.
\newblock {\em Archive for Rational Mechanics and Analysis}, 241(3):1529--1573,
  2021.

\bibitem{Carrillo2014}
J.~A. Carrillo, Y.-P. Choi, and M.~Hauray.
\newblock The derivation of swarming models: Mean-field limit and {W}asserstein
  distances.
\newblock In Adrian Muntean and Federico Toschi, editors, {\em Collective
  Dynamics from Bacteria to Crowds: An Excursion Through Modeling, Analysis and
  Simulation}, pages 1--46. Springer Vienna, Vienna, 2014.

\bibitem{cdfls}
J.~A. Carrillo, M.~Di~Francesco, A.~Figalli, T.~Laurent, and D.~Slepčev.
\newblock Global-in-time weak measure solutions and finite-time aggregation for
  nonlocal interaction equations.
\newblock {\em Duke Mathematical Journal}, 156, 02 2011.

\bibitem{choi_fagioli_iorio}
Y.-P. Choi, S.~Fagioli, and V.~Iorio.
\newblock Small inertia limit for coupled kinetic swarming models.
\newblock {\em Journal of Nonlinear Science}, 35(39), 2025.

\bibitem{choi_jeong_fractional}
Y.-P. Choi and I.-J. Jeong.
\newblock Classical solutions for fractional porous medium flow.
\newblock {\em Nonlinear Analysis}, 210:112393, 2021.

\bibitem{choi_j}
Y.-P. Choi and I.-J. Jeong.
\newblock Global-in-time existence of weak solutions for
  {V}lasov-{M}anev-{F}okker-{P}lanck system.
\newblock {\em Kinetic and Related Models}, 16(1):41--53, 2023.

\bibitem{difrafag}
M.~Di~Francesco and S.~Fagioli.
\newblock Measure solutions for non-local interaction {PDE}s with two species.
\newblock {\em Nonlinearity}, 26:2777, 2013.

\bibitem{difra_iorio}
M.~Di~Francesco and V.~Iorio.
\newblock A system of continuity equations with nonlocal interactions of
  {M}orse type.
\newblock {\em Communications on Pure and Applied Analysis}, 24(8):1381--1405,
  2025.

\bibitem{difra_i_schm}
M.~Di~Francesco, V.~Iorio, and M.~Schmidtchen.
\newblock The approximation of the quadratic porous medium equation via
  nonlocal interacting particles subject to repulsive {M}orse potential.
\newblock {\em SIAM Journal on Mathematical Analysis}, 57(5):4631--4679, 2025.

\bibitem{HauJab}
M.~Hauray and P.-E. Jabin.
\newblock {N}-particles approximation of the {V}lasov equations with singular
  potential.
\newblock {\em Archive for Rational Mechanics and Analysis}, 183(3):489--524,
  2007.

\bibitem{mainini_volzone}
Y.~Huang, E.~Mainini, J.~L. Vázquez, and B.~Volzone.
\newblock Nonlinear aggregation-diffusion equations with {R}iesz potentials.
\newblock {\em Journal of Functional Analysis}, 287(2):110465, 2024.

\bibitem{jko}
R.~Jordan, D.~Kinderlehrer, and F.~Otto.
\newblock The variational formulation of the {F}okker-{P}lanck equation.
\newblock {\em SIAM J. Math. Anal.}, 29(1):1--17, 1998.

\bibitem{lis_main_seg}
S.~Lisini, E.~Mainini, and A.~Segatti.
\newblock A gradient flow approach to the porous medium equation with
  fractional pressure.
\newblock {\em Archive for Rational Mechanics and Analysis}, 227:567--606,
  2018.

\bibitem{matthes2009family}
D.~Matthes, R.~J. McCann, and G.~Savaré.
\newblock A family of nonlinear fourth order equations of gradient flow type.
\newblock {\em Communications in Partial Differential Equations},
  34(11):1352--1397, 2009.

\bibitem{mogilner}
A.~Mogilner and L.~Edelstein-Keshet.
\newblock A non-local model for a swarm.
\newblock {\em J. Math. Biol.}, 38(6):534--570, 1999.

\bibitem{Golse2016}
A.~Muntean, J.~Rademacher, and A.~Zagaris.
\newblock {\em On the Dynamics of Large Particle Systems in the Mean Field
  Limit}.
\newblock Springer International Publishing, Cham, 2016.

\bibitem{okubo}
A.~Okubo and S.~A. Levin.
\newblock {\em Diffusion and ecological problems: modern perspectives},
  volume~14 of {\em Interdisciplinary Applied Mathematics}.
\newblock Springer-Verlag, New York, second edition, 2001.

\bibitem{santambrogio_book}
F.~Santambrogio.
\newblock {\em Optimal Transport for Applied Mathematicians: Calculus of
  Variations, PDEs, and Modeling}.
\newblock Progress in Nonlinear Differential Equations and Their Applications.
  Springer International Publishing, 2015.

\bibitem{Serfaty_vazq}
S.~Serfaty and J.~L. V{\'a}zquez.
\newblock A mean field equation as limit of nonlinear diffusions with
  fractional laplacian operators.
\newblock {\em Calculus of Variations and Partial Differential Equations},
  49:1091 -- 1120, 2013.

\bibitem{topaz}
C.~M. Topaz and A.~L. Bertozzi.
\newblock Swarming patterns in a two-dimensional kinematic model for biological
  groups.
\newblock {\em SIAM J. Appl. Math.}, 65(1):152--174, 2004.

\bibitem{villani}
C.~Villani.
\newblock {\em Topics in Optimal Transportation}, volume~58 of {\em Graduate
  Studies in Mathematics}.
\newblock American Mathematical Society, Providence, RI, 2003.

\end{thebibliography}

\end{document}